\documentclass[a4paper,10pt,english]{smfart}

\usepackage[utf8x]{inputenc}
\RequirePackage[T1]{fontenc}
\RequirePackage{amsfonts,latexsym,amssymb}
\RequirePackage[frenchb]{babel}
\usepackage[all]{xy}
\addto\extrasfrenchb{\mathbfl@nonfrenchitemize
\mathbfl@nonfrenchspacing}
\usepackage{geometry}
\newtheoremstyle{theorem}{11pt}{11pt}{\slshape}{}{\mathbfseries}{.}{.5em}{}
\newtheoremstyle{note}{11pt}{11pt}{}{}{\mathbfseries}{.}{.5em}{}

\theoremstyle{plain}
  \newtheorem{theoreme}{Theorem}[section]
  \newtheorem{proposition}[theoreme]{Proposition}
  
  \newtheorem{lemme}[theoreme]{Lemma}
  \newtheorem{corollaire}[theoreme]{Corollary}
  \newtheorem{conj}[theoreme]{Conjecture}

\theoremstyle{definition}
  \newtheorem{definition}[theoreme]{Definition}

\theoremstyle{remark}

  \newtheorem{remarque}[theoreme]{Remark}
  \newtheorem{remarks}[theoreme]{Remarks}

\usepackage{url}
\usepackage{todonotes}

\let\mathcal\mathcal

\let\mathbf\mathbf

\def\C{{\mathbf C}}

\def\R{{\mathbf R}}
\def\O{{\mathcal O}}

\def\zp{{\Z_p}}

\def\qp{{\Q_p}}

\def\p1{{\mathbf P}^1}
\def\qp{\mathbf{Q}_p}
\def\zp{\mathbf{Z}_p}
\def\z{\mathbf{Z}}

\def\cp{\mathbf{C}_p}
\def\q{\mathbf{Q}}

\def\epsilon{\varepsilon}

\begin{document}
\title[De Rham cohomology over the Fargues-Fontaine curve]{Overconvergent relative de Rham cohomology over the Fargues-Fontaine curve}
\author{Arthur-C\'esar Le Bras}
\address{Mathematisches Institut, Universit\"at Bonn, Endenicher Aller 60, 53115 Bonn, Zimmer 4.0.25}
\email{lebras@math.uni-bonn.de}

\begin{abstract}
We explain how to construct a cohomology theory on the category of separated quasi-compact smooth rigid spaces over $\cp$ (or more general base fields), taking values in the category of vector bundles on the Fargues-Fontaine curve, which extends (in a suitable sense) Hyodo-Kato cohomology when the rigid space has a semi-stable proper formal model over the ring of integers of a finite extension of $\qp$, therefore proving a conjecture of Scholze \cite{SICM}. This cohomology theory factors through the category of rigid analytic motives of Ayoub.
\end{abstract}

\maketitle

\stepcounter{tocdepth}
{\Small
\tableofcontents
}

\section{Introduction}

Fargues and Fontaine introduced a few years ago the \textit{fundamental curve of $p$-adic Hodge theory}, now simply called the Fargues-Fontaine curve. It turned out to be indeed a fundamental geometric object : most of $p$-adic Hodge theory can be rephrased using the curve, and recent spectacular work of Fargues and Scholze shows how the curve and several closely related objects allow to reformulate geometrically the local Langlands correspondence. The reader is referred to \cite{FarguesRio} (and references therein) for a nice overview of these recent developments. Let us here simply briefly recall the definition of the Fargues-Fontaine curve. The curve has two incarnations : an \textit{algebraic} one and an \textit{adic} one\footnote{There is also a very simple but important description of the diamond associated to the curve, due to Scholze \cite{SW}.}, which are both useful. We give the latter definition, which is shorter. Let $\cp$ be the completion of an algebraic closure of $\qp$ and $\cp^{\flat}$ its tilt, a complete algebraically closed field of characteristic $p$. One starts with the analytic adic space
\[ Y = \mathrm{Spa}(W(\O_{\cp^{\flat}}),W(\O_{\cp^{\flat}})) \backslash V(p[p^{\flat}]), \]
where $p^{\flat} \in \O_{\cp^{\flat}}$ has first coordinate $p$, i.e. is such that $(p^{\flat})^{\sharp} =p$. One can think of the space $Y$ as the mixed characteristic version of the punctured open unit disk (now \og in the variable $p$ \fg{}). 

The ring of analytic functions on $Y$ is called $B$ ; it is the Fr\'echet algebra obtained by completing the algebra
\[ B^b := \left\{ \sum_{n \gg -\infty} [x_n] p^n , x_n \in \cp^{\flat}, (x_n)_n ~ \mathrm{bounded} \right\} = W(\O_{\cp^{\flat}})[1/p,1/[p^{\flat}]] \]
of \og meromorphic functions along the divisors $p=0$ and $[p^{\flat}]=0$ \fg{} with respect to the norms $\| \cdot \|_{\rho}$, for all $0 < \rho <1$, defined by
\[ \| \sum_n [x_n]p^n \|_{\rho} = \sup_n |x_n| \rho^n. \]
There is a Frobenius operator $\varphi$ on $Y$, whose action on functions is given by the formula
\[ \sum_n [x_n] p^n \mapsto \sum_n [x_n^p] p^n. \]
The action of $\varphi$ on $Y$ is properly discontinuous and one can then define the Fargues-Fontaine curve $X$ (for the local field $\qp$) as the quotient :
\[ X := Y/\varphi^{\z}. \]
The set of classical points of $X$ is in bijection with the set of untilts of $\cp^{\flat}$ in characteristic $0$ (modulo the action of Frobenius) : if $x \in |X|$, the residue field $C_x$ of $X$ at $x$ is an algebraically closed complete valued extension of $\qp$, with $C_x^{\flat} \simeq \cp^{\flat}$. In particular, there is a point called $\infty$ on $X$, such that $C_{\infty}=\cp$. 

In fact, the exact same construction works if one replaces $\cp$ by any algebraically closed complete valued extension $C$ of $\qp$. In this case, one gets adic spaces which will be called $Y_{C^{\flat}}$ and $X_{C^{\flat}}$ and which have exactly the same properties.
\\

It turns out that the Fargues-Fontaine curve $X$ looks like the non-archimedean analog of the twistor projective line $\widetilde{\mathbf{P}}_{\R}^1$, which is the quotient of the complex projective line by the antipodal involution :
\[ \widetilde{\mathbf{P}}_{\R}^1 := \mathbf{P}_{\C}^1 / z \sim -\bar{z}^{-1}. \]
This is a real variety without any real point (exactly like the Fargues-Fontaine curve, whose residue fields at classical points are all algebraically closed). The analogy between the curve and $\widetilde{\mathbf{P}}_{\R}^1$ has so far no precise mathematical meaning, but there are some interesting similarities between these two objects. Here are a few examples :

(a) The set of $G$-torsors over $\widetilde{\mathbf{P}}_{\R}^1$ (resp. $X$) is in bijection with Kottwitz's set $B(G)$ (as defined in \cite{kottwitz}), for $G$ reductive over $\R$ (resp. $\qp$). In the non-archimedean setting, this is the main result of \cite{gtorseurs}, see also \cite{anschuetz} ; in the archimedean setting, this can be proved for example using the strategy of \cite{anschuetz}, as explained to us by Ansch\"utz.

(b) The function field of $\widetilde{\mathbf{P}}_{\R}^1$ is conjectured to be (C1) (Lang \cite{lang}) ; similarly, Fargues (\cite{gtorseurs}) conjectures the function field $\mathrm{Frac}(B[1/t]^{\varphi=1})$ of $X$ to be (C1).

(c) The variety $\widetilde{\mathbf{P}}_{\R}^1$ is the Brauer-Severi variety of Hamilton's quaternion algebra over $\R$ ; similarly, it was conjectured by Fargues that $X$ should be in some sense the Brauer-Severi variety of a division algebra $\mathcal{C}$ introduced by Colmez in \cite{colmez}.
\\

In this paper, we study another instance of this analogy. It is a nowadays classical result of Simpson (\cite[\S 2]{simpson}, see also \cite[\S 3.1-3.4]{mochi}) that the datum of a real pure Hodge structure can be encoded as a $U(1)$-equivariant semi-stable vector bundle on $\widetilde{\mathbf{P}}_{\R}^1$. Let $Z$ be a smooth projective variety over $\C$. The $U(1)$-equivariant vector bundle attached to the cohomology of degree $i$ of $Z$ is obtained by modifying the $U(1)$-equivariant vector bundle
\[ H_{\infty}^i(Z):=H_{\rm Betti}^i(Z,\R) \otimes \O_{\widetilde{\mathbf{P}}_{\R}^1} \]
at the point $\infty \in \widetilde{\mathbf{P}}_{\R}^1$ lying above $0$ and $\infty$ (which is fixed by the $U(1)$-action), using the Hodge filtration on $H_{\rm Betti}^i(Z,\C)$. Similarly, if $Z$ is proper and smooth over $K$, a complete discretely valued extension of $\qp$ with perfect residue field, one can modify the $\mathcal{G}_K$-equivariant vector bundle (obtained by descending to $X$) $H_{\mathcal{FF}}^i(Z):=\mathcal{E}(D_{\rm pst}(H_{\mathrm{\acute{e}t}}^i(Z_{C},\qp))$, where $C=\widehat{\bar{K}}$, using the Hodge filtration on $D_{\rm pst} \otimes K$. This modification is the $\mathcal{G}_K$-equivariant vector bundle $H_{\mathrm{\acute{e}t}}^i(Z_{\cp},\qp) \otimes_{\qp} \O_X$\footnote{The slightly confusing point in this analogy is that it does not match the two trivial vector bundles...}. The first construction relies on the equivalence between filtrations on $V$ and $\C^*$-equivariant lattices in $V\otimes_{\C} \C((t))$, if $V$ is a finite dimensional $\C$-vector space ; the second on the equivalence between filtrations on $V$ and $\mathcal{G}_K$-equivariant lattices in $V\otimes_{K} B_{\rm dR}$, if $V$ is a finite dimensional $K$-vector space. 
\\

Over the field of complex numbers, the vector bundle $H_{\infty}^i(Z)$ can still be defined more generally, by the same formula :
\[ H_{\infty}^i(Z):=H_{\rm Betti}^i(Z,\R) \otimes \O_{\widetilde{\mathbf{P}}_{\R}^1}, \]
which makes sense whenever the Betti/de Rham cohomology is finite dimensional. Our goal in this text is to show that we can also extend the construction of the vector bundle $H_{\mathcal{FF}}^i(Z)$ outside the proper case, and give a geometric definition of this vector bundle. More precisely, let $K$ be a complete valued field containing $\qp$ (not assumed to be discretely valued), $C$ the completion of an algebraic closure of $K$ (for example, $K$ may be a finite extension of $\qp$, or $\cp$) with residue field $k$, and $\mathcal{G}_K=\mathrm{Gal}(C/K)$ the absolute Galois group of $K$. We explain how to construct a cohomology theory
\[ Z \mapsto \mathcal{FF}(Z), \]
on the category of smooth quasi-compact rigid spaces\footnote{All the rigid spaces we consider are assumed to be separated and taut.} defined over $K$, taking values in the bounded derived category $D^b(\mathrm{Coh}_{X_{C^{\flat}}})$ of coherent sheaves on the Fargues-Fontaine curve $X_{C^{\flat}}$, such that if $C$ is the completion of an algebraic closure of $W(k)[1/p]$, $Z$ is a smooth quasi-compact rigid space of dimension $n$ and $i\geq 0$, the coherent sheaf
\[ H_{\mathcal{FF}}^i(Z):=H^i(\mathcal{FF}(Z)) \]
is a vector bundle, which vanishes if $i<0$ or $i>2n$. Under this assumption on $C$, this cohomology theory satisfies several properties explained below and thus confirms Conjecture 6.4 of \cite{SICM} (see also \cite[Conj. 1.13]{FarguesRio}).  

In particular, still under this assumption on $C$, one gets an isocrystal $H_{\rm isoc}^i(Z)$ attached to $Z$, for all $i$ (see Corollary \ref{isocristal} ; there is some caveat).
\\

A few remarks are in order.

\begin{remarque} \label{propercase}
Assume that $K$ is a complete discretely valued extension of $\qp$ with perfect residue field and that $Z$ is smooth and \textit{proper}. We show that for all $i\geq 0$, $H_{\mathcal{FF}}^i(Z)$ is the $\mathcal{G}_K$-equivariant vector bundle on $X_{C^{\flat}}$ attached as in \cite[Ch. 10]{FF} to the $D_{\rm pst}$ of the \'etale cohomology $H_{\mathrm{\acute{e}t}}^i(Z_{C}, \qp)$, using comparison theorems in $p$-adic Hodge theory.
\end{remarque}

\begin{remarque} \label{semistablecase}
When the smooth rigid variety $Z$ over $K$ is such that $Z_C$ has a smooth quasi-compact formal model $\mathfrak{Z}$ over $\O_C$, we expect $H_{\mathcal{FF}}^i(Z)$ to be the vector bundle attached to the isocrystal given by the rigid cohomology of the special fiber of $\mathfrak{Z}$. This implies in particular that this isocrystal is independent of the choice of the formal model. We prove this statement when $\mathfrak{Z}$ is proper, which is easy using the results of Bhatt-Morrow-Scholze \cite{BMS} (in the proper case, the independence of the isocrystal is already known : when $K$ is a complete discretely valued extension of $\qp$ with perfect residue field, it is a corollary of a deep theorem : Fontaine's $C_{\rm crys}$ conjecture ; in general it was recently proved in \cite{BMS}). 

One can also ask a similar question when $\mathfrak{Z}$ is only assumed to be semi-stable, and answer it in the proper case, using the recent work of \u{C}esnavi\u{c}ius-Koshikawa \cite{CJ}.

Note that our construction of $H_{\mathcal{FF}}^i(Z)$ is entirely on the generic fiber and does not involve any choice of formal model.
\end{remarque}

This cohomology theory is quite simple to describe : if $Z$ is as before, $\mathcal{FF}(Z)$ comes by the equivalence between (the derived category of) vector bundles on $X$ and finite projective $\varphi$-modules over $B$ (which are the same as $\varphi$-equivariant vector bundles on $Y$) from an \og overconvergent version \fg{} of the complex $R\Gamma(Z_C, L\eta_t R\nu'_* \mathbb{B}_Z)$, $(\nu'_*,\nu^{' -1})$ being the morphism from the pro-\'etale topos of $Z$ to the \'etale topos of $Z_C$. There are essentially two things to prove about these cohomology groups : a "finiteness" and a "torsion-freeness" result. For both, the key point is to show that this cohomology theory factors through the category of rigid analytic motives of Ayoub (\cite{Ayoub}). This reduces to prove finiteness in the case where $Z$ is smooth and \textit{proper}, where several tools are available. This strategy was inspired by Vezzani's paper \cite{Vezzani}, which studies overconvergent de Rham cohomology using the motivic formalism. This works for any $C$. If $C$ is the completion of an algebraic closure of $W(k)[1/p]$, we also get torsion-freeness, using a trick : any quasi-compact motive over $C$ descends to a motive defined over a discretely valued extension of $\qp$ with perfect residue field (a statement which is false for smooth quasi-compact rigid spaces !) ; therefore, we can use some $p$-adic Hodge theory.  


Let us end this introduction by saying that there should exist a more geometric construction of $\mathcal{FF}(Z)$, making the analogy with Betti cohomology more transparent and explaining the title of this text. Let us start by (trivially !) reformulating the construction of $H_{\infty}^i(Z)$ in the complex case. It is obtained by descending to $\tilde{\mathbf{P}}_{\R}^1$ the vector bundle $H_{\rm Betti}^i(Z,\C) \otimes_{\C} \O_{\mathbf{A}_{\C}^1}$, which, in the diagram, 
\[ \xymatrix{ & Z \times \mathbf{A}_{\C}^1 \ar_p[ld] \ar^q[rd] \\ Z && \mathbf{A}_{\C}^1} \]
is simply
\[ R^iq_* (\Omega_{Z \times \mathbf{A}_{\C}^1/\mathbf{A}_{\C}^1}^{\bullet}). \]
When $Z$ is moreover projective, the twistor is obtained by descending to $\tilde{\mathbf{P}}_{\R}^1$ the sheaf
\[ R^iq_* ((\Omega_{Z \times \mathbf{A}_{\C}^1/\mathbf{A}_{\C}^1}^{\bullet},zd)), \]
where $z$ is the coordinate on $\mathbf{A}_{\C}^1$ (this is Deligne's reinterpretation of Simpson's construction). 

Similarly, in the $p$-adic case, we can look at the diagram of diamonds :
\[ \xymatrix{ & Z^{\diamond} \times \mathrm{Spa}(\qp)^{\diamond} \ar_p[ld] \ar^q[rd] \\ Z^{\diamond} && Y^{\diamond} } \] 
as $Y^{\diamond} = \mathrm{Spa}(\cp)^{\diamond} \times \mathrm{Spa}(\qp)^{\diamond}$. There should exist a complex of \'etale sheaves $"\Omega_{Z^{\diamond} \times \mathrm{Spa}(\qp)^{\diamond}/Y}^{\dagger,\bullet}"$ such that the pull-back of $H_{\mathcal{FF}}^i(Z)$ to $Y$ is
\[ R^iq_* ("\Omega_{Z^{\diamond} \times \mathrm{Spa}(\qp)^{\diamond}/Y}^{\dagger,\bullet}"). \]
In joint work in progress with Alberto Vezzani, we show how to give a meaning to this hypothetical overconvergent relative de Rham complex\footnote{The relation between the pro-\'etale cohomology of $\mathbb{B}$ and Deligne's $\lambda$-connections was also noticed by Liu and Zhu, cf, \cite[Rem. 3.2]{LZ}.}. Still, the point of the present paper is that, even if we don't know what this complex is, $Y$ being Stein, to construct $H_{\mathcal{FF}}^i(Z)$, it is enough to know what
\[ Rp_*  ("\Omega_{Z^{\diamond} \times \mathrm{Spa}(\qp)^{\diamond}/Y}^{\dagger,\bullet}") \]
is : this is precisely (the overconvergent version of) $L\eta_t R\nu'_* \mathbb{B}_Z$. 
\\

\textit{Plan of the text.} Sections \ref{Motives} and \ref{Period sheaves} are devoted to some remainders and some fundamental results about the category of rigid motives, pro-\'etale period sheaves and the d\'ecalage functors. The key technical statement of this text (showing that our cohomology theory $\mathcal{F}$ factors through the category of motives) is proved in Section \ref{Coeur}. We define $\mathcal{FF}$ and prove that its cohomology groups are vector bundles on the Fargues-Fontaine curve when $C$ is the completion of an algebraic closure of some discretely valued field in Section \ref{Construction}. Finally, we briefly explain in Section \ref{relation} how our constructions are related to other cohomology theories in some cases : \'etale, crystalline, syntomic and pro-\'etale cohomologies. 
\\

\textit{Conventions and notations.} The field $K$ is a complete valued field containing $\qp$ ; $C$ the completion of an algebraic closure of $K$, with residue field $k$ (in particular $K$ may very well be $C$ itself !) ; $\mathbf{D}$ is the closed unit disk over $\qp$ and $\mathbf{D}^{\dagger}$ its overconvergent analogue. We fix a compatible system of $p^n$-th roots of unity $\epsilon \in \cp^{\flat}$, with $\epsilon^{(0)}=1$, $\epsilon^{(1)}\neq 1$, and let $t=\log([\epsilon]) \in B$ be \og Fontaine's $2i\pi$ \fg{}. All rigid spaces are assumed to be separated and taut.
\\

\textit{Acknowledgments.} The idea of this paper originated from a conversation with Peter Scholze in Bonn in February 2017. I would like to thank him heartily for this, as well as for his numerous very helpful suggestions and comments. Some of the key cohomological computations of this text were started when redacting Chapter 2 of \cite{theseaclb} as part of my PhD thesis and I thank Laurent Fargues for everything he taught me about period sheaves and the Fargues-Fontaine curve. Special thanks go to Alberto Vezzani for his interest and the time he took to answer my many naive questions about motives. I am indebted to Wieslawa Nizio\l{} who pointed out a gap in a first version of this text. Finally, I would also like to thank K\c{e}stutis \u{C}esnavi\u{c}ius and Teruhisa Koshikawa for interesting conversations and Eugen Hellmann for an invitation to talk about this work.

\section{Motives of (overconvergent) rigid analytic varieties} \label{Motives}
For what follows, the main reference is \cite{Vezzani}. We take $\Lambda=\qp$ as a coefficient ring ; it will be implicit in the notation. In all this section, let $F$ be a characteristic $0$ complete valued field with a non archimedean valuation of rank $1$ and residue characteristic $p$. We recall the definition of the category of rigid motives (and overconvergent rigid motives), which can be thought of as the Verdier quotient of the derived category of \'etale sheaves of $\qp$-vector spaces on the category of smooth rigid varieties over $F$, obtained by requiring that all the projections maps $(Z \times \mathbf{D}_F)[i] \to Z[i]$ are invertible, for $Z$ any rigid smooth space over $F$ and $i\in \z$, and we state two results which will be useful later on (Theorem \ref{ayoub} and Proposition \ref{2lim} below).  

Let $\mathrm{Rig}_F$ (resp. $\mathrm{Aff}_F$) be the category of rigid spaces over $F$ (resp. affinoid rigid spaces over $F$), and let $\mathrm{Rig}_F^{\dagger}$ (resp. $\mathrm{Aff}_F^{\dagger}$) be the category of overconvergent rigid spaces over $F$ (resp. overconvergent affinoid rigid spaces), in the sense of Große-Klönne \cite{grosse}. The natural functor $\mathrm{Rig}_F^{\dagger} \to \mathrm{Rig}_F$, $Z \mapsto \hat{Z}$ is denoted by $\ell$. Let $\mathrm{RigSm}_F$ be the subcategory of smooth rigid analytic varieties over $F$ and $\mathrm{RigSm}_F^{\dagger}$ be the subcategory of smooth overconvergent rigid varieties over $F$ (an overconvergent rigid space $Z$ is said to be smooth if the associated rigid space $\hat{Z}$ is). 

If $Z$ is affinoid over $F$, a \textit{presentation of an overconvergent structure} on $Z$ is by definition a pro-affinoid variety $\varprojlim Z_h$, where $Z$ and $Z_h$, for all $h\geq 1$, are rational subspaces of $Z_1$ with $Z \Subset Z_h \Subset Z_1$ (for all $h>1$) and such that this system is coinitial among rational subspaces of $Z_1$ strictly containing $Z$. A \textit{morphism} of presentations is a morphism of pro-objects. 

\begin{proposition} \label{a22}
If $Z=\mathrm{Spa}(\hat{R},\hat{R}^{\circ})$ is affinoid over $F$, and $\varprojlim Z_h$ is a presentation of an overconvergent structure on $Z$, then $R:= \varinjlim \O(Z_h)$ is dense in $\hat{R}$ and the functor $\varprojlim Z_h \mapsto \mathrm{Spa}^{\dagger}(R)$ induces an equivalence between the category of presentations of overconvergent structures on affinoid spaces and the category $\mathrm{Aff}_F^{\dagger}$. 
\end{proposition}
\begin{proof}
See \cite[Prop. A.22]{Vezzani}.
\end{proof}

In what follows, "presheaf" always means "presheaf of $\qp$-vector spaces". If $Z \in \mathrm{RigSm}_F$ (resp. $\mathrm{RigSm}_F^{\dagger}$), we will denote by $\qp(Z)$ the presheaf $Z' \mapsto \qp \mathrm{Hom}_{\mathrm{RigSm}_F}(Z',Z)$ on $\mathrm{RigSm}_F$ (resp. $Z' \mapsto \qp \mathrm{Hom}_{\mathrm{RigSm}_F^{\dagger}}(Z',Z)$ on $\mathrm{RigSm}_F^{\dagger}$). 

\begin{definition}
Let $\mathcal{S}_{\mathrm{\acute{e}t}}$ be the class of maps $\mathcal{F} \to \mathcal{F}'$ in the category of complexes of presheaves over $\mathrm{RigSm}_F$ or $\mathrm{RigSm}_F^{\dagger}$ inducing isomorphims on the \'etale sheaves attached to the cohomology presheaves of these complexes. Let $\mathcal{S}_{\mathbf{D}}$ (resp. $\mathcal{S}_{\mathbf{D}^{\dagger}}$) be the set of all maps $\qp(\mathbf{D}_F \times Z)[i] \to \qp(Z)[i]$ (resp. maps $\qp(\mathbf{D}_F^{\dagger} \times Z)[i] \to \qp(Z)[i])$, as $Z$ varies in $\mathrm{RigSm}_F$ (resp. $\mathrm{RigSm}_F^{\dagger}$) and $i \in \z$. Let $\mathcal{S}_{\mathrm{\acute{e}t},\mathbf{D}}$ (resp. $\mathcal{S}_{\mathrm{\acute{e}t},\mathbf{D}^{\dagger}}$) be the union of the two classes. The homotopy category of the left Bousfield localization of the projective model category on the category of presheaves on $\mathrm{RigSm}_F$ (resp. $\mathrm{RigSm}_F^{\dagger}$) with respect to the class $\mathcal{S}_{\mathrm{\acute{e}t},\mathbf{D}}$ (resp. $\mathcal{S}_{\mathrm{\acute{e}t},\mathbf{D}^{\dagger}}$) will be denoted $\mathrm{RigMot}_F$ (resp. $\mathrm{RigMot}_F^{\dagger}$) : it is the category of \textit{rigid analytic motives over $F$} (resp. the category of \textit{overconvergent rigid analytic motives over $F$}). 
\end{definition}
If $Z \in \mathrm{RigSm}_F$, we will again denote by $\qp(Z)$ the associated motive in $\mathrm{RigMot}_F$. 
\\

These two categories are actually equivalent (\cite[Th. 4.23]{Vezzani}) :

\begin{theoreme} \label{vezza}
The functors $(L\ell^*,R\ell_*)$ induce quasi-inverse equivalences of triangulated monoidal categories :
\[ \mathrm{RigMot}_F \simeq \mathrm{RigMot}_F^{\dagger}. \]
\end{theoreme}

\begin{definition}
Let $\mathrm{RigMot}_F^{\rm comp}$ be the full triangulated subcategory of $\mathrm{RigMot}_F$ formed by compact objects (those objects $\mathcal{C}$ such that the functor $\mathrm{Hom}_{\mathrm{RigMot}_F}(\mathcal{C},\cdot)$ commutes with small sums).
\end{definition}

We will make a crucial use of the following two results.

\begin{theoreme} \label{ayoub}
The subcategory $\mathrm{RigMot}_F^{\rm comp}$ coincides with the saturated triangulated subcategory of $\mathrm{RigMot}_F$ generated by the motives $\qp(Z)[d]$, where $Z$ runs among proper smooth rigid varieties\footnote{Even : analytifications of smooth projective varieties, but this statement will be enough for us.} and $d$ runs in $\z$.
\end{theoreme}
\begin{proof}
See \cite[Th. 2.5.35]{Ayoub}.
\end{proof}

\begin{proposition} \label{2lim}
Assume that $C$ is the completion of an algebraic closure of $W(k)[1/p]$. The natural functors induce equivalences :
\[ \mathrm{RigMot}_{C}^{\rm comp} \simeq 2-\varinjlim ~ \mathrm{RigMot}_F^{\rm comp}, \]
where $F$ runs among finite extensions of $W(k)[1/p]$.
\end{proposition}
\begin{proof}
See \cite[Lem. 5.21, Lem. 5.23]{Vezzani2}.
\end{proof}

\begin{remarque}
The analogous statement at the level of quasi-compact smooth rigid spaces is completely false. This shows the interest of working with motives.

Let us briefly give the rough idea behind the proof of this statement. Let $Z$ be a quasi-compact smooth rigid space over $C$. One can easily find a finite cover of $Z$ by affinoid spaces which descend to a finite extension $F$ of $W(k)[1/p]$ (using, for example, Elkik's approximation theorem \cite{elkik}). Then one can approximate the gluing isomorphisms over $C$ by morphisms defined over $F$ (up to enlarging $F$) ; by Lemma \ref{homotopie} below, these morphisms are homotopic to the original ones and thus define the same morphisms in the category of motives. Therefore, the motives of the models over $F$ of our affinoids "glue" to a motive over $F$.

We used the following simple
\begin{lemme} \label{homotopie}
Let $A$ be an affinoid algebra, with its norm $\| \cdot \|$, and $f : A\to A$ an automorphism satisfying $\| f - \mathrm{Id} \| < p^{-1/(p-1)}$. There exists a morphism $H : A \to A \langle T \rangle$ such that $\iota_0 \circ H= \mathrm{Id}$ and $\iota_1 \circ H=f$, $\iota_0$ (resp. $\iota_1$) being the evaluation at $0$ (resp. $1$). 
\end{lemme}
\begin{proof}
(D'apr\`es \cite{vdp}.) Let 
\[ D = \log(f) = \sum_n (-1)^n \frac{(\mathrm{Id}-f)^{n+1}}{n+1}. \]
This is a derivation of $A$ which satisfies $\| D \| < p^{-1/(p-1)}$. It suffices to set, for $a \in A$ :
\[ H(a) = \sum_n \frac{D^n(a)}{n!} T^n \]
to get the sought-after homotopy. 
\end{proof}
\end{remarque}

\textit{Notation.} From now on, we will simply write $\mathrm{RigMot}$ (resp. $\mathrm{RigMot}^{\dagger}$) for the category $\mathrm{RigMot}_C$ (resp. $\mathrm{RigMot}_C^{\dagger}$).

\section{Pro-\'etale period sheaves and d\'ecalage functors} \label{Period sheaves}

In this section, we recall or prove some results about period sheaves and d\'ecalage functors, which will be used in the next section.

\begin{definition}
Let $f : Z' \to Z$ be a morphism of analytic adic spaces over $C$. The morphism $f$ is said to be \textit{affinoid pro-\'etale} if $Z'=\mathrm{Spa}(S,S^+)$, $Z=\mathrm{Spa}(R,R^+)$ are affinoid and if $Z' = \varprojlim Z_i \to Z$ can be written as cofiltrant projective limit of \'etale morphisms $Z_i \to Z$, with $Z_i=\mathrm{Spa}(S_i,S_i^+)$ affinoid. The morphism is said to be \textit{pro-\'etale} if it is affinoid pro-\'etale locally on source and target.

Let $Z$ be an analytic adic space over $C$. The \textit{(small) pro-\'etale site} of $Z$ is the Grothendieck topology on the category of $f : Z' \to Z$ pro-\'etale, with $Z'$ perfectoid over $C$, for which a family of morphisms $\{ f_i : Z'_i \to Z', i\in I \}$ is a covering if each $f_i$ is pro-\'etale and if or each quasi-compact open $U \subset Z'$, there exist a finite subset $J\subset I$ and for each $i\in J$ a quasi-compact open $U_i \subset Z'_i$, such that $U=\cup_{i\in J} f_i(U_i)$. 
\end{definition}

Let $Z$ be a rigid space over $C$. There is no morphism of sites $Z_{\mathrm{pro\acute{e}t}} \to Z_{\mathrm{\acute{e}t}}$\footnote{There is such a morphism if one uses instead the old pro-\'etale site of \cite{SHodge} but the definition of the pro-\'etale site we take is different.}, but there is still a morphism $(\nu'_*,\nu^{' -1})$ form the pro-\'etale topos of $Z$ to the \'etale topos of $Z$. 

\begin{definition}
Let $Z$ be a rigid space over $C$. We define the following pro-\'etale sheaves on $Z$. 
\begin{itemize}
\item The completed integral structure sheaf $\widehat{\O}_Z^+ = \varprojlim_r \O_Z^+/p^r$ and its tilted analogue $\widehat{\O}_{Z^{\flat}}^+ = \varprojlim_{\varphi} \O_Z^+/p$.
\item The sheaf $\mathbb{A}_{\rm inf}= W(\widehat{\O}_{Z^{\flat}}^+)$. One has a morphism of sheaves $\theta : \mathbb{A}_{\rm inf} \to \widehat{\O}_Z^+$, which extends to $\theta : \mathbb{A}_{\rm inf}[1/p] \to \widehat{\O}_Z$.  
\item The sheaf $\mathbb{B}_{\rm dR}^+ = \varprojlim_k \mathbb{A}_{\rm inf}[1/p]/(\ker(\theta))^k$. It is endowed with a filtration defined by $\mathrm{Fil}^i \mathbb{B}_{\rm dR}^+ = \ker(\theta)^i$. 
\item Let $t$ be a generator of $\mathrm{Fil}^1 \mathbb{B}_{\rm dR}^+$ (such an element exists locally for the pro-\'etale topology, is unique up to a unit and is not a zero divisor). Let $\mathbb{B}_{\rm dR} = \mathbb{B}_{\rm dR}^+[1/t]$ and $\mathrm{Fil}^i \mathbb{B}_{\rm dR} = \sum_{j\in \z} t^{-j} \mathrm{Fil}^{i+j} \mathbb{B}_{\rm dR}^+$. 
\end{itemize}
\end{definition}  

\begin{definition} \label{aibi}
Let $Z$ be an adic space over $C$. Let $I =[a,b] \subset ]0,1[$ be a compact subinterval, with $a, b\in p^{\q}$. If $\alpha, \beta \in \O_{C^{\flat}}$ are such that $|\alpha|=a$, $|\beta|=b$, one defines :
\[ \mathbb{A}_I = \underset{n}\varprojlim ~ (\mathbb{A}_{\rm inf}[[\alpha]/p,p/[\beta]])/p^n  \quad ; \quad \mathbb{B}_I=\mathbb{A}_I[1/p]. \]
Let :
\[ \mathbb{B} = \underset{I \subset ]0,1[}\varprojlim ~ \mathbb{B}_I. \]
\end{definition}

The pro-\'etale cohomology of these period sheaves can be simply described in the proper case.

\begin{proposition} \label{proprepro\'et}
Let $Z$ be a smooth proper rigid space over $C$. Let
\[ \mathbb{M} \in \{ \mathbb{B}_{\rm dR}^+, \mathbb{B}_{\rm dR}, \mathbb{B}_I, \mathbb{B}, \mathbb{B}_I[1/t], \mathbb{B}[1/t] \},  \]
and accordingly
\[ M \in \{ B_{\rm dR}^+, B_{\rm dR}, B_I, B, B_I[1/t], B[1/t] \}. \]
One has a quasi-isomorphism
\[  R\Gamma_{\mathrm{\acute{e}t}}(Z, \qp) \otimes_{\qp} M \simeq R\Gamma_{\mathrm{pro\acute{e}t}}(Z, \mathbb{M}). \]
Moreover, the cohomology groups $H_{\mathrm{\acute{e}t}}^i(Z,\qp)$ are finite dimensional $\qp$-vector spaces, for all $i\geq 0$. 
\end{proposition}
\begin{proof}
As $Z$ is proper, thus quasi-compact, it is enough to prove the statement when $\mathbb{M} \in \{ \mathbb{B}_{\rm dR}^+, \mathbb{B}_I, \mathbb{B} \}$. Using Lemma \ref{replete} below, the statement for $\mathbb{M}=\mathbb{B}$ is easily deduced from the statement for $\mathbb{B}_I$, for all compact subintervals $I \subset ]0,1[$. The sheaf $\mathbb{B}_I$ is obtained by $p$-adically completing $\mathbb{A}_{\rm inf}[[\alpha]/p,p/[\beta]]$, for $\alpha, \beta$ as in Definition \ref{aibi}, and inverting $p$. So it is enough to prove that for any $x \in \O_{C^{\flat}}$, $0< |x|<1$, and any $m>0$, 
\[  R\Gamma_{\mathrm{\acute{e}t}}(Z, \zp/p^m) \otimes_{\zp} A_{\rm inf}/([x]) \simeq^a R\Gamma_{\mathrm{pro\acute{e}t}}(Z, \mathbb{A}_{\rm inf}/(p^m,[x])). \]
This is a corollary of the proof of \cite[Th. 8.4]{SHodge}. The case $\mathbb{M}=\mathbb{B}_{\rm dR}^+$ is also treated in \cite[Th. 8.4]{SHodge}. Finaly, the last statement is \cite[Th. 3.1]{ScholzeCDM}.
\end{proof}

\begin{lemme} \label{replete}
Let $Z$ be a smooth proper rigid space over $C$. For any $i >0$, one has
\[ R^i  \underset{I} \varprojlim ~ \mathbb{B}_I =0, \]
when $I$ runs among compact subintervals of $]0,1[$, with extremities in $p^{\q}$.
\end{lemme}
\begin{proof}
We apply \cite[Lem. 3.18]{SHodge}, by taking as a basis the set of all affinoid perfectoid spaces pro-\'etale over $Z$. Condition (ii) of \textit{loc. cit.} follows by acyclicity of $\mathbb{B}_I$ on affinoid perfectoid spaces. Condition (i) comes from the fact that for any $U \in Z_{\mathrm{pro\acute{e}t}}$ affinoid perfectoid, and any $I \subset J$, the morphism $\mathbb{B}_J(U) \to \mathbb{B}_I(U)$ is continuous morphism of Banach spaces with dense image.  
\end{proof}

In general, if $Z$ is a smooth quasi-compact space, these cohomology groups do not have good finiteness properties. Two issues are to be overcome. The first one is classical : as for de Rham cohomology, one needs to make everything overconverge. But this is still not enough : for example, the \og overconvergent pro-\'etale cohomology \fg{} of $\mathbb{B}_{\rm dR}^+$ on the closed unit disk over $C$ would be isomorphic to $C \langle T \rangle^{\dagger}$, viewed as a $B_{\rm dR}^+$-module via Fontaine's map $\theta : B_{\rm dR}^+ \to C$. This $B_{\rm dR}^+$-module has torsion and is not of finite type ! Bhatt, Morrow and Scholze \cite{BMS} understood how to modify these period sheaves to remedy this problem, using a simple but magical operation discovered by Deligne and already used by Berthelot and Ogus. We will briefly recall its definition and some of its properties.
\\

Let $A$ be a ring and $f\in A$, which is not a zero divisor. 

\begin{definition}
If $K^{\bullet}$ is a complex of $A$-modules such that $K^i$ is $f$-torsion torsion free for all $i$, one defines a new complex $\eta_f K^{\bullet}$ by
\[ (\eta_{f} K^{\bullet})^j = \{ x \in f^{j} K^j, dx \in f^{j+1} K^{j+1} \}. \]
The functor $\eta_f$ sends quasi-isomorphisms to quasi-isomorphisms ; its extension\footnote{Here one uses the fact that any complex of $A$-modules is quasi-isomorphic to a complex of $f$-torsion free $A$-modules.}  to the derived category is called $L\eta_f$.
\end{definition}

The functor $L\eta_f$ kills $f$-torsion in cohomology. It of course does nothing when $f$ is invertible in $A$.

The action of the d\'ecalage functor on some Koszul complexes is simple to describe. This is actually an important point, as many proofs rely on a local study where pro-\'etale cohomology groups can be written as Koszul complexes.
\begin{definition}
Let $M$ be an abelian group and $h_1,\dots, h_d$ be commuting endomorphisms of $M$. Let $K_M(h_1,\dots,h_d)$ be the Koszul complex :
\[ M \to \bigoplus_{1 \leq i \leq d} M \to \bigoplus_{1 \leq i_1<i_2 \leq d} M \to \dots \to \bigoplus_{1\leq i_1<\dots < i_k \leq d} M \to \dots \]
where the differential on $M$ in position $i_1< \dots < i_k$ to $M$ in position $j_1<\dots, < j_{k+1}$ is non zero only if $\{i_1,\dots,i_k \} \subset \{j_1, \dots,j_{k+1}\}$ and equals $(-1)^{m-1} h_m$ in this case, $m$ being the unique index between $1$ and $k+1$ such that $j_m \notin \{i_1,\dots,i_k\}$. 
\end{definition}

\begin{proposition} \label{koszul}
Let $A$ be a ring, $g_1,\dots,g_n \in A$ and $f\in A$ which are not zero divisors. Let $M^{\bullet}$ be a complex of $f$-torsion-free $A$-modules.

$\mathrm{(i)}$ If $f$ divides all the $g_i$,
\[ \eta_f (M^{\bullet} \otimes_A K_A(g_1,\dots,g_n)) = \eta_f M^{\bullet} \otimes_A K_A(g_1/f,\dots,g_n/f). \]

$\mathrm{(ii)}$ If there exists $i$ such that $g_i$ divides $f$, $\eta_f (M^{\bullet} \otimes_A K_A(g_1,\dots,g_n))$ is acyclic.
\end{proposition}
\begin{proof}
See \cite[Lem. 7.9]{BMS}.
\end{proof}

One can actually define d\'ecalage functors in greater generality.
\begin{definition}
Let $\delta : \z \to \mathbf{N}$ be a non decreasing function. If $K^{\bullet}$ is a complex of $f$-torsion free $A$-modules, let $\eta_{\delta,f} K^{\bullet}$ be the complex defined by
\[ (\eta_{\delta,f} K^{\bullet})^i = \{ x \in f^{\delta(i)} K^i, dx \in f^{\delta(i+1)} K^{i+1} \}. \]
The functor $\eta_{\delta,f}$ preserves quasi-isomorphisms (\cite[Prop. 8.19]{BO}) and extends to a functor $L\eta_{\delta,f}$ on the derived category of $A$-modules.
\end{definition}

Another useful fact is the following

\begin{lemme} \label{induction}
Let $A$ be a ring, $f \in A$ a non zero divisor, and $K^{\bullet}$ a bounded complex of $A$-modules, such that $K^j$ is $f$-torsion free for all $j$. Let, for $i\geq 0$, $\delta_i$ be the function defined by 
\[ \delta_i(j) = \max(0,j-i). \]
Note that $K_{\delta_0,f}^{\bullet} = \eta_f K^{\bullet}$ and that $K_{\delta_i,f}^{\bullet}=K^{\bullet}$ for any $i$ big enough. Then one has, for all $i \geq 0$, a triangle in $D(A)$
\[ K_{\delta_{i},f}^{\bullet} \longrightarrow K_{\delta_{i+1},f}^{\bullet} \longrightarrow \tau^{\geq i+1} (K^{\bullet}/f) \overset{+1} \longrightarrow. \] 
\end{lemme}

These definitions and results extend to a more general framework (\cite[\S 6.1]{BMS}) : if $(T,\O_T)$ is a ringed topos, $f\in \O_T$ generates an invertible ideal and $K^{\bullet}$ is a complex of $\O_T$-modules without $f$-torsion, one can define $\eta_{\delta,f} K^{\bullet}$, for any non decreasing $\delta$, as before and one can show that $\eta_{\delta,f}$ extends to the derived category $D(\O_T)$ of $\O_T$-modules.

In particular, let $Z$ be a rigid space over $C$ and $\mathbb{M}$ and $M$ be as in Proposition \ref{proprepro\'et}. Specializing to the case where $T$ is the topos of \'etale sheaves of $M$-modules on $Z$, $\O_T=M$ and $f=t$, one obtains a complex $L\eta_t R\nu'_* \mathbb{M}$ in the derived category of \'etale sheaves of $M$-modules on $Z$. 

\begin{proposition} \label{commutation}
For any smooth affinoid rigid space $Z$ over $C$, and 
\[ \mathbb{M} \in \{ \mathbb{B}_{\rm dR}^+, \mathbb{B}_{\rm dR}, \mathbb{B}_I, \mathbb{B}_I[1/t] \},  \]
the natural map
\[ L\eta_t R\Gamma_{\mathrm{pro\acute{e}t}}(Z, \mathbb{M}) \to R\Gamma_{\mathrm{\acute{e}t}}(Z,L\eta_t R\nu'_* \mathbb{M}) \]
is a quasi-isomorphism.
\end{proposition}
\begin{proof}
The statement is of course empty when $\mathbb{M} \in \{ \mathbb{B}_{\rm dR}, \mathbb{B}_I[1/t] \}$, as $t$ in invertible and the complexes are bounded. We can thus assume that $\mathbb{M} \in  \{ \mathbb{B}_{\rm dR}^+, \mathbb{B}_I \}$, $I$ being such that the annulus of radius $I$ contains at least one zero of $t$, in which case $\mathbb{M}/t$ is isomorphic to a finite number of copies of $\widehat{\O}$. Using Lemma \ref{induction}, applied both to $L\eta_t R\Gamma_{\mathrm{pro\acute{e}t}}(Z, \mathbb{M})$ and $L\eta_t R\nu'_* \mathbb{M}$, we see that it is enough to show that for all $i\geq 0$ :
\[ H_{\mathrm{\acute{e}t}}^j(Z,\tau^{\geq i} R\nu'_* \widehat{\O}) \simeq H_{\mathrm{pro\acute{e}t}}^j(Z,\widehat{\O}) \]
if $j\geq i$ and that
\[ H_{\mathrm{\acute{e}t}}^j(Z,\tau^{\geq i} R\nu'_* \widehat{\O}) =0 \]
if $j<i$ (we used the fact that $\mathbb{M}/t$ is a finite product of copies of $\widehat{\O}$). To do so, we use the spectral sequence
\[ H_{\mathrm{\acute{e}t}}^{j-k}(Z,H^{k} (\tau^{\geq i} R\nu'_* \widehat{\O})) \Longrightarrow H_{\mathrm{\acute{e}t}}^j(Z,\tau^{\geq i} R\nu'_* \widehat{\O}) \] 
and Scholze's computation (\cite[Prop. 3.23]{ScholzeCDM}  :
\[ H^{k} (\tau^{\geq i} R\nu'_* \widehat{\O}) =  \left\{ \begin{array}{ll} 0 & \mbox{if } k < i \\
H^{k} (R\nu'_* \widehat{\O})= \Omega^k & \mbox{if } k \geq i. \end{array} \right. \]
Similarly, we have a spectral sequence 
\[ H_{\mathrm{\acute{e}t}}^{j-k}(Z,\Omega^k) \Longrightarrow H_{\mathrm{pro\acute{e}t}}^j(Z,\widehat{\O}). \]
The differences between the abutments of these two spectral sequences could thus a priori come from the terms $H^{j-k}(Z,\Omega^k)$, for $k<i$. But then as $j\geq i$ we have $j-k>0$ and $Z$ being affinoid, these terms are zero.
\end{proof}

\begin{proposition} \label{cj}
Let $\mathfrak{Z}$ be a semi-stable quasi-compact formal scheme over $\O_C$, with generic fiber $Z$, and $(\nu_*,\nu^{-1})$ be the morphism from the pro-\'etale topos of $Z$ to the \'etale topos of $\mathfrak{Z}$. Let $I$ be a compact interval contained in $]p^{-1},1[$ and $\mathbb{A}_I$ be the sheaf defined in Definition \ref{aibi}. The natural map\footnote{Let $\lambda$ be the morphism from the \'etale site of $Z$ to the \'etale site of $\mathfrak{Z}$. The map of the proposition is induced by the map 
\[ L\eta_{\mu} R\nu_* \mathbb{A}_{\rm inf} = L\eta_{\mu} R\lambda_* R\nu'_* \mathbb{A}_{\rm inf} \to R\lambda_* L\eta_{\mu} R\nu'_* \mathbb{A}_{\rm inf} \to R\lambda_* L\eta_{\mu} R\nu'_* \mathbb{A}_{I}. \]} 
\[ \left( R\Gamma_{\mathrm{\acute{e}t}}(\mathfrak{Z}, L\eta_{\mu} R\nu_* \mathbb{A}_{\rm inf}) \hat{\otimes}_{A_{\rm inf}} A_I \right) [1/p] \to R\Gamma_{\mathrm{\acute{e}t}}(Z, L\eta_{\mu} R\nu'_* \mathbb{B}_I) \simeq R\Gamma_{\mathrm{\acute{e}t}}(Z, L\eta_t R\nu'_* \mathbb{B}_I) \]
is an isomorphism (the last isomorphism comes from the fact that $t$ and $\mu$ are equal up to a unit in $A_I$), the completed tensor product on the left being for the $p$-adic topology.
\end{proposition}
\begin{proof}
One can reduce to the case where $\mathfrak{Z}=\mathrm{Spf}(R)$, with $R$ small as in \cite{CJ} : there exists an \'etale map
\[ \square : \mathrm{Spf}(R) \to \mathrm{Spf}(\O_C \langle T_0,\dots,T_r,T_{r+1}^{\pm 1},\dots T_d^{\pm 1} \rangle / (T_0 \dots T_r - p^q)), \]
for some $d\geq 0$, some $0 \leq r \leq d$ and some $q\in \q_{>0}$. Let $R_{\infty}$ be the completed tensor product of $R$ with
\[ \O_C \langle T_0^{1/p^{\infty}},\dots,T_r^{1/p^{\infty}},T_{r+1}^{\pm 1/p^{\infty}},\dots T_d^{\pm 1/p^{\infty}} \rangle / (T_0^{1/p^{\infty}} \dots T_r^{1/p^{\infty}} - [(p^{1/p^{\infty}})^q]) \]
over $\O_C \langle T_0,\dots,T_r,T_{r+1}^{\pm 1},\dots T_d^{\pm 1} \rangle / (T_0 \dots T_r - p^q)$. As $\square$ is \'etale, one can lift $R$ to a $(p,\mu)$-adically complete $A_{\rm inf}$-algebra $A_{\rm inf}(R)$ formally \'etale over
\[ A_{\rm inf}\langle X_0,\dots,X_r,X_{r+1}^{\pm 1},\dots X_d^{\pm 1} \rangle / (X_0 \dots X_r - [(p^{1/p^{\infty}})^q]). \]
Let $A_I(R)= A_{\rm inf}(R) \hat{\otimes}_{A_{\rm inf}} A_I$. As $W(\mathfrak{m}^{\flat})$ is invertible on both sides, we can replace pro-\'etale cohomology groups by continuous cohomology groups. Then there are two steps :
\\
 
\textit{Step 1.} One can decompose $\mathbb{A}_{\rm inf}(Z)=A_{\rm inf}(R_{\infty}^{\flat})$ as
\[ A_{\rm inf}(R_{\infty}^{\flat}) = A_{\rm inf}(R) \oplus N_{\infty}, \]
where $N_{\infty} = A_{\rm inf}(R_{\infty}^{\flat})^{\rm nonint}$. One has
\[ A_{I}(R_{\infty}^{\flat})^{\rm nonint} = N_{\infty} \hat{\otimes}_{A_{\rm inf}} A_I, \]
the completion being $p$-adic. Assume that for all $i$,
\[ H_{\rm cont}^i( \zp^d, N_{\infty} \hat{\otimes}_{A_{\rm inf}} A_I) \simeq H_{\rm cont}^i(\zp^d, N_{\infty}) \hat{\otimes}_{A_{\rm inf}} A_I. \]
As we know that $\mu$ kills $H_{\rm cont}^i(\zp^d, N_{\infty})$ (\cite{BMS} in the smooth case, \cite[Prop. 3.25]{CJ} in the semi-stable case), $\mu$ kills the right hand side and thus also the left hand side. To prove the above isomorphism, one proceeds as in \cite{CJ} :

\begin{lemme} \label{vanishtor}
Let $A$ be a ring endowed with the $f$-adic topology for some $f\in A$, $A'$ an $A$-module complete for the $f$-adic topology. Consider the following condition $(*)$ on an $A$-module $L$ : we require that for each $j, n$ the map
\[ \mathrm{Tor}_j^A(L, A'/f^{n'}) \to \mathrm{Tor}_j^A(L, A'/f^n) \]
vanishes for some $n'>n$. Then if $K^{\bullet}$ is a bounded complex of $A$-modules, if each $K^i$ and $H^i(K^{\bullet})$ satisfy $(*)$, then for every $i$,
\[ H^i(K^{\bullet} \hat{\otimes}_A A') = H^i(K^{\bullet}) \hat{\otimes}_A A', \]
where the completion is $f$-adic.
\end{lemme}
\begin{proof}
See \cite[Lem. 3.30]{CJ}.
\end{proof}
One can then redo the proof of \cite[\S 3.28]{CJ}, working with the ring $A_I$ instead of the ring $A_{\rm cris}^{(m)}$ used in \textit{loc. cit.}. Indeed, the only properties of $A_{\rm cris}^{(m)}$ used in \cite{CJ} to make the argument work is that the system of ideals $(p^n A_{\rm cris}^{(m)})_{n \geq 0}$ and $(\{ x \in A_{\rm cris}^{(m)}, \mu x \in p^n A_{\rm cris}^{(m)} \})_{n \geq 0}$ are intertwined, to get the key statement (3.26.3) (used to check Point (iii) of their Lemma 3.30). Here the same is true for $A_I$ : indeed $p^k/\mu$ is in $A_I$ for $k$ big enough, so the two topologies are the same. 
\\

\textit{Step 2.} It thus only remains to show that the map
\[ L\eta_{\mu} R\Gamma_{\rm cont}(\zp^d, A_{\rm inf}(R)) \hat{\otimes}_{A_{\rm inf}} A_I  \to L\eta_{\mu} R\Gamma_{\rm cont}(\zp^d, A_I(R)) \]
is an isomorphism, but this is easy since we know that both sides identify with $q$-de Rham complexes.
\end{proof}

\begin{remarque} \label{egalomega}
Let $\mathfrak{Z}=\mathrm{Spf}(R)$, with $R$ small. If $I$ is as in the proposition, we have that
\[ R\Gamma_{\mathrm{\acute{e}t}}(Z, L\eta_t R\nu'_* \mathbb{B}_I) \simeq \Omega_{B_I(R)/B_I}^{\bullet}. \] 
As $\varphi(t)=pt$, the Frobenius induces an isomorphism between $\varphi^* R\Gamma(Z, L\eta_t R\nu'_* \mathbb{B}_I)$ and $R\Gamma(Z, L\eta_t R\nu'_* \mathbb{B}_{\varphi(I)})$ and so one deduces from the above proof that for \textit{any} compact interval $I \subset ]0,1[$,
\[  R\Gamma_{\mathrm{\acute{e}t}}(Z, L\eta_t R\nu'_* \mathbb{B}_I) \simeq \Omega_{B_I(R)/B_I}^{\bullet}. \]
\end{remarque}

\section{Motives attached to period sheaves}  \label{Coeur}

The next proposition is the key ingredient for the main result of this section, Proposition \ref{factormotives}.
\begin{proposition} \label{contractile}
Let $Z=\mathrm{Spa}(R,R^{\circ})$ be a smooth affinoid rigid space over $C$, with an \'etale map
\[ \square : Z \to \mathbf{T}^d= \mathrm{Spa} (C\langle T_1^{\pm 1},\dots, T_d^{\pm 1} \rangle, \O_C\langle T_1^{\pm 1},\dots, T_d^{\pm 1} \rangle) \]
which factorizes as a composite of rational embeddings and finite \'etale maps. For all $n\geq 1$, let $\mathbf{D}_n=\mathrm{Spa}(\qp \langle p^{1/n}T \rangle)$ be the closed disk of radius $p^{1/n}$. Let $\mathbb{M}, M$ be as in Proposition \ref{proprepro\'et}. The projection maps $Z \times \mathbf{D}_n \to Z$ induce an isomorphism :
\[  R\Gamma_{\mathrm{\acute{e}t}}(Z, L\eta_t R\nu'_* \mathbb{M}) \simeq \underset{n} \varinjlim ~ R\Gamma_{\mathrm{\acute{e}t}}(Z \times \mathbf{D}_{n,C}, L\eta_t R\nu'_* \mathbb{M}). \]
\end{proposition}
\begin{proof}
Everything being quasi-compact, it suffices to prove the statement when $\mathbb{M} \in \{\mathbb{B}_I, \mathbb{B}, \mathbb{B}_{\rm dR}^+ \}$. The choice of the chart $\square$ determines a perfectoid pro-\'etale algebra $R \to R_{\infty}$, i.e. a pro-\'etale perfectoid cover $Z_{\infty}=\mathrm{Spa}(R_{\infty},R_{\infty}^{\circ}) \to Z$. The pro-\'etale cohomology of $\mathbb{M}$ can be computed as the cohomology of the Koszul complex
\[ K_{\mathbb{M}(Z_{\infty})}^{\bullet}=K_{\mathbb{M}(Z_{\infty})}(\gamma_1-1,\dots,\gamma_d-1), \]
$\gamma_1,\dots,\gamma_d$ being generators of $\zp^d$. 

We will start by proving the following result.

\begin{lemme} \label{boule}
Let $\mathbb{M} \in \{\mathbb{B}_I,\mathbb{B}_{\rm dR}^+/t^r \}$ ($r \geq 1$). One has:
\[ R\Gamma_{\mathrm{\acute{e}t}}(Z \times \mathbf{D}_C, L\eta_t R\nu'_* \mathbb{M}) \simeq L\eta_t K_{\mathbb{M}(Z_{\infty})}^{\bullet} \hat{\otimes}_{\qp} \Omega^{\bullet}(\mathbf{D}), \]
the completed tensor product being a completed tensor product of complexes of Tate $\qp$-algebras.
\end{lemme}

\begin{proof}
By Proposition \ref{commutation}, we can first replace $R\Gamma_{\mathrm{\acute{e}t}}(Z \times \mathbf{D}_C, L\eta_t R\nu'_* \mathbb{M})$ by $L\eta_t R\Gamma_{\mathrm{pro\acute{e}t}}(Z \times \mathbf{D}_C, \mathbb{M})$.

Let $\widetilde{\mathbf{D}}=\mathrm{Spa}(\qp \langle T^{1/p^{\infty}} \rangle)$ be the perfectoid unit disk. The morphism $\widetilde{\mathbf{D}}_C \to \mathbf{D}_C$ is not pro-\'etale, but it is quasi-pro-\'etale, that is pro-\'etale locally for the pro-\'etale topology, which will be enough. The cohomology of $\mathbb{M}$ on $Z_{\infty} \times_C \widetilde{\mathbf{D}}_C$ is trivial. Hence we can compute the pro-\'etale cohomology of $\mathbb{M}$ on $Z \times_C \mathbf{D}_C$ as the Cech cohomology of the covering
\[ Z_{\infty} \times_C \widetilde{\mathbf{D}}_C \to Z \times \mathbf{D}. \]

Let $\mathrm{ev}$ be the evaluation map at $T=0$. For all $k\geq 1$, the fiber product $k$-times $\widetilde{\mathbf{D}}_C^{(k)} = \widetilde{\mathbf{D}}_C \times_{\mathbf{D}_C} \dots \times_{\mathbf{D}_C} \widetilde{\mathbf{D}}_C$ is affinoid perfectoid, with affinoid algebra $(A_{k},A_k^+)$, where $A_{k}=A_{k}^+[1/p]$ and $A_{k}^+$ is the ring formed by the elements $f \in \mathcal{C}^0(\zp^{k-1}, \O_C \langle T^{1/p^{\infty}} \rangle)$ such that $\mathrm{ev} \circ f \in \mathcal{C}^0(\zp^{k-1},\O_C)$ is a constant function : for a proof of this assertion, see \cite[Lem. 2.3.28]{theseaclb}. Therefore, $\widehat{\O}_{X^{\flat}}^+(\widetilde{\mathbf{D}}_C^{(k)} \times_C Z_{\infty}^{(k)})$ is the ring formed by functions $f \in \mathcal{C}^0(\z_p^{k-1} \times \zp^{d(k-1)}, R_{\infty}^{(k),\flat, +}\langle T^{1/p{\infty}} \rangle)$, such that for any $\underline{x} \in \zp^{d(k-1)}$, $\mathrm{ev} \circ f(\cdot,\underline{x}) \in \mathcal{C}^0(\z_p^{k-1}, R_{\infty}^{(k),\flat, +})$ is a constant function.

Thus the $k$-th term of the Cech complex we consider $\mathbb{M}(\widetilde{\mathbf{D}}_C^{(k)} \times_C Z_{\infty}^{(k)})$ is the ring formed by the elements $$f \in \mathcal{C}^0(\zp^{k-1} \times \zp^{d(k-1)}, \mathbb{M}(Z_{\infty}^{(k)}) \langle X^{1/p^{\infty}} \rangle)$$\footnote{When $M=\mathbb{B}_{\rm dR}^+$, $\mathbb{B}_{\rm dR}^+(Z_{\infty}^{(k)}) \langle X^{1/p^{\infty}} \rangle = \varprojlim_r (\mathbb{B}_{\rm dR}^+(Z_{\infty}^{(k)})/t^r) \langle X^{1/p^{\infty}} \rangle$.}, with $X=[T^{\flat}]$, such that for all $\underline{x} \in \zp^{d(k-1)}$, $\mathrm{ev}' \circ f(\cdot,\underline{x}) \in \mathcal{C}^0(\zp^{k-1},\mathbb{M}(Z_{\infty}^{(k)}))$ is a constant function, $\mathrm{ev}'$ being the evaluation at $X=0$. One has an exact sequence :
\[ 0 \to  \mathcal{C}^0(\zp^{k-1} \times \zp^{d(k-1)}, \mathrm{Ker}(\mathrm{ev}')) \to \mathbb{M}(\widetilde{\mathbf{D}}_C^{(k)} \times_C Z_{\infty}^{(k)}) \to \mathcal{C}^0(\zp^{d(k-1)},\mathbb{M}(Z_{\infty}^{(k)})) \to 0, \]
the second arrow being $f \mapsto \mathrm{ev}' \circ f$.

Let's look at the subcomplex of the Cech complex whose $k$-th term is given by
\[ \mathcal{C}^0(\zp^{k-1} \times \zp^{d(k-1)}, \mathrm{Ker}(\mathrm{ev}')).  \]
It is just the standard complex of continuous cochains for the group $\zp^{d+1}$ acting on $\mathrm{Ker}(\mathrm{ev}')$ ; it is isomorphic to the Koszul complex $K_{\mathrm{Ker}(\mathrm{ev}')}(\gamma_0-1,\dots,\gamma_{d}-1)$. Let
\[ \mathrm{Ker}(\mathrm{ev}')^{\rm int} =X \mathbb{M}(Z_{\infty}) \langle X \rangle. \]
Then the Koszul complex of $\mathrm{Ker}(\mathrm{ev}')$ decomposes as the sum of 
\[ K_{\mathrm{Ker}(\mathrm{ev}')^{\rm int}}(\gamma_0-1,\gamma_1-1,\dots,\gamma_d-1) \]
and of the completed direct sum 
\[ \widehat{\bigoplus}_{a \in \mathbf{N}[1/p] \backslash \mathbf{N}} K_{\mathrm{Ker}(\mathrm{ev}')^{\rm int}}(\gamma_0[\epsilon]^a-1,\gamma_1-1,\dots,\gamma_d-1). \]
We can then argue exactly as in the proof of \cite[Lem. 9.6]{BMS} to prove that multiplication by $t$ on this completed direct sum is homotopic to zero (here we use the fact that $t$ and $[\epsilon]-1$ differ by a unit, by hypothesis on $I$), and therefore killed by $L\eta_t$.  Consequently, the only remaining term is :
\[ L\eta_t K_{\mathrm{Ker}(\mathrm{ev}')^{\rm int}}(\gamma_0-1,\gamma_1-1,\dots,\gamma_d-1) = L\eta_t \left( K_{X M\langle X \rangle}(\gamma_0-1) \hat{\otimes}_M K_{\mathbb{M}(Z_{\infty})}^{\bullet} \right), \]
where the completed tensor product makes sense as a completed tensor product of $\qp$-Tate algebras. Note that $\gamma_0-1$ and multiplication by $t$ commute as endomorphisms of $M\langle X \rangle$ and that the latter divides the former. We can thus copy the argument of \cite[Lem. 7.9]{BMS} to get an isomorphism
\[  L\eta_t \left( K_{X M\langle X \rangle}(\gamma_0-1) \hat{\otimes}_M K^{\bullet} \right) \simeq K_{X M\langle X \rangle}\left(\frac{\gamma_0-1}{t} \right) \hat{\otimes}_M L\eta_t K_{\mathbb{M}(Z_{\infty})}^{\bullet}. \]
By \cite[Lem. 12.3]{BMS}, this complex is also quasi-isomorphic to :
\[ K_{X M\langle X \rangle}\left(X\frac{d}{dX} \right) \hat{\otimes}_M L\eta_t K_{\mathbb{M}(Z_{\infty})}^{\bullet}, \] 
which identifies to the subcomplex 
\[ \left( X \cdot \qp\langle X \rangle \overset{d} \longrightarrow \qp \langle X \rangle \cdot dX \right) \hat{\otimes}_{\qp} L\eta_t K_{\mathbb{M}(Z_{\infty})}^{\bullet} \]
of $\Omega^{\bullet}({\mathbf{D}}) \hat{\otimes}_{\qp} L\eta_t K_{\mathbb{M}(Z_{\infty})}^{\bullet}$.

All in all, we get a quasi-isomorphism :
\[ R\Gamma_{\mathrm{\acute{e}t}}(Z \times \mathbf{D}_C, L\eta_t R\nu'_* \mathbb{M}) \simeq L\eta_t K_{\mathbb{M}(Z_{\infty})}^{\bullet} \hat{\otimes}_{\qp} \Omega^{\bullet}(\mathbf{D}), \]
as wanted.
\end{proof}

Now we prove Proposition \ref{contractile}, using Lemma \ref{boule}. Let $\mathbb{M}= \mathbb{B}_I$ or $\mathbb{M}= \mathbb{B}_{\rm dR}^+/t^r$, $r\geq 1$.

Choose $n \geq 1$. The $i$-th term of the complex $L\eta_t K_{\mathbb{M}(Z_{\infty})}^{\bullet} \hat{\otimes}_{\qp} \Omega^{\bullet}(\mathbf{D}_n)$ is :
\[ \left(L\eta_t K_{\mathbb{M}(Z_{\infty})}^{\bullet} \hat{\otimes}_{\qp} \Omega^{\bullet}(\mathbf{D}_n) \right)^i = \left((L\eta_t K_{\mathbb{M}(Z_{\infty})}^{\bullet})^i \hat{\otimes}_{\qp} \O(\mathbf{D}_n) \right) \oplus \left( (L\eta_t K_{\mathbb{M}(Z_{\infty})}^{\bullet})^{i-1} \hat{\otimes}_{\qp} \Omega^{1}(\mathbf{D}_n) \right). \]
Any element $x$ in it can thus be written as 
\[ \sum_{k \geq 0} D_k(x) T^k + \sum_{k \geq 0} L_k(x) T^k dT, \]
with $D_k(x) \in (L\eta_t K_{\mathbb{M}(Z_{\infty})}^{\bullet})^i$, $L_k(x) \in (L\eta_t K_{\mathbb{M}(Z_{\infty})}^{\bullet})^{i-1}$ and $p^{-k/n} D_k(x), p^{-k/n} L_k(x) \to 0$ when $k \to +\infty$. In particular, for any such $x$, the series
\[ \sum_{k \geq 0} \frac{L_k(x)}{k+1} T^{k+1} \]
converges to an element $L(x) \in (L\eta_t K_{\mathbb{M}(Z_{\infty})}^{\bullet} \hat{\otimes}_{\qp} \Omega^{\bullet}(\mathbf{D}_{n+1}))^i$, and this defines a morphism $L$ of the complex $L\eta_t K_{\mathbb{M}(Z_{\infty})}^{\bullet} \hat{\otimes}_{\qp} \Omega^{\bullet}(\mathbf{D}_n)$ into $L\eta_t K_{\mathbb{M}(Z_{\infty})}^{\bullet} \hat{\otimes}_{\qp} \Omega^{\bullet}(\mathbf{D}_{n+1})$. Taking the inductive limit over $n$ we get a morphism of the complex
\[ \underset{n} \varinjlim ~ \left( L\eta_t K_{\mathbb{M}(Z_{\infty})}^{\bullet} \hat{\otimes}_{\qp} \Omega^{\bullet}(\mathbf{D}_n) \right) \]
into itself, again denoted $L$.
One easily checks the identity :
\[ d \circ L + L \circ d = \mathrm{Id} - D_0, \]
which proves that the map $D_0$ is a homotopy inverse to the natural map 
\[ L\eta_t K_{\mathbb{M}(Z_{\infty})}^{\bullet} \to \underset{n} \varinjlim ~ \left( L\eta_t K_{\mathbb{M}(Z_{\infty})}^{\bullet} \hat{\otimes}_{\qp} \Omega^{\bullet}(\mathbf{D}_n) \right). \]
This concludes the proof of Proposition \ref{contractile} when $\mathbb{M}=\mathbb{B}_I$ or $\mathbb{B}_{\rm dR}^+/t^r$. Let's now treat the case $\mathbb{M}=\mathbb{B}$.

\begin{lemme} \label{44}
The natural map 
\[ L\eta_t R\nu'_* \mathbb{B} \to R ~ \underset{I} \varprojlim ~  L\eta_t R\nu'_* \mathbb{B}_I \]
is an isomorphism.
\end{lemme}
\begin{proof}
We use Lemma \ref{induction}. This reduces us to prove that the map 
\[ R\nu'_* \mathbb{B} \to R ~ \underset{I} \varprojlim ~  R\nu'_* \mathbb{B}_I \]
is an isomorphism, which follows from Lemma \ref{replete}.  
\end{proof}

As derived projective limits commute with $R\Gamma$ and with derived completed tensor products, we have
\[ \underset{n} \varinjlim ~ R\Gamma_{\mathrm{\acute{e}t}}(Z \times \mathbf{D}_{n,C}, L\eta_t R\nu'_* \mathbb{B}) = \underset{n} \varinjlim ~ R \underset{I} \varprojlim ~ R\Gamma_{\mathrm{\acute{e}t}}(Z \times \mathbf{D}_{n,C}, L\eta_t R\nu'_* \mathbb{B}_I) \] 
\[ = \underset{n} \varinjlim ~ R \underset{I} \varprojlim ~ \left( L\eta_t K_{\mathbb{B}_I(Z_{\infty})}^{\bullet} \hat{\otimes}_{\qp} \Omega^{\bullet}(\mathbf{D}_n) \right), \]
the last equality by Lemma \ref{boule}. As the two limits involved are on different factors of the completed tensor product, they commute and so :
\[ \underset{n} \varinjlim ~ R\Gamma_{\mathrm{\acute{e}t}}(Z \times \mathbf{D}_{n,C}, L\eta_t R\nu'_* \mathbb{B}) =  R \underset{I} \varprojlim ~ \underset{n} \varinjlim ~ \left( L\eta_t K_{\mathbb{B}_I(Z_{\infty})}^{\bullet} \hat{\otimes}_{\qp} \Omega^{\bullet}(\mathbf{D}_n) \right) \]
\[ = R \underset{I} \varprojlim ~ \underset{n} \varinjlim ~ R\Gamma_{\mathrm{\acute{e}t}}(Z \times \mathbf{D}_{n,C}, L\eta_t R\nu'_* \mathbb{B}_I), \]
again by applying Lemma \ref{boule}. Similarly, 
\[ R\Gamma_{\mathrm{\acute{e}t}}(Z, L\eta_t R\nu'_* \mathbb{B}) = R ~ \underset{I} \varprojlim ~ R\Gamma_{\mathrm{\acute{e}t}}(Z, L\eta_t R\nu'_* \mathbb{B}_I). \]
Therefore, Proposition \ref{contractile} for $\mathbb{M}=\mathbb{B}$ follows from Proposition \ref{contractile} for $\mathbb{M}=\mathbb{B}_I$ that we just proved. 

The exact same argument works to deduce Proposition \ref{contractile} for $\mathbb{M}=\mathbb{B}_{\rm dR}^+$ from Proposition \ref{contractile} for $\mathbb{M}=\mathbb{B}_{\rm dR}^+/t^r$ (in this case the analog of Lemma \ref{44} can also be deduced from \cite[Lem. 6.20]{BMS}). 
\end{proof}

Let
\[ \mathbb{M} \in \{ \mathbb{B}_{\rm dR}^+, \mathbb{B}_{\rm dR}, \mathbb{B}_I, \mathbb{B}, \mathbb{B}_I[1/t], \mathbb{B}[1/t] \} \]
($I \subset ]0,1[$ is a compact interval). We fix once and for all an injective resolution $\mathcal{I}_{\mathbb{M}}^{\bullet}$ of the pro-\'etale sheaf $\mathbb{M}$ viewed as an element of the pro-\'etale topos of the category $\mathrm{RigMot}$.  

\begin{definition}
Let $\mathcal{F}_{\mathbb{M}}$ be the complex of \'etale sheaves on the category of overconvergent affinoid spaces over $C$, such that for $Z$ overconvergent affinoid, $\mathcal{F}_{\mathbb{M},Z}$ corresponds to the inverse system
\[ \underset{h} \varinjlim ~ \eta_t \nu'_* \mathcal{I}_{\mathbb{M},Z_h}^{\bullet}, \]
where $\varprojlim Z_h$ is a presentation of the overconvergent structure $Z$ on $\hat{Z}$, through the equivalence of Proposition \ref{a22}. 
\end{definition}

\begin{proposition} \label{factormotives}
The complex of \'etale sheaves $\mathcal{F}_{\mathbb{M}}$ is $\mathbf{D}_C^{\dagger}$-local. 
\end{proposition}
\begin{proof}
Let $\varprojlim \mathbf{D}_h$ be the standard overconvergent presentation of the dagger structure $\mathbf{D}^{\dagger}$ on $\mathbf{D}$ by closed disks $\mathbf{D}_h$ of radius $p^{1/h}$. Let $Z$ be quasi-compact smooth rigid space over $C$. We need to check that for all $i$, the natural map 
\[ H^i(Z, \mathcal{F}_{\mathbb{M},Z}) \to H^i(Z \times \mathbf{D}_C^{\dagger}, \mathcal{F}_{\mathbb{M},Z \times \mathbf{D}_C^{\dagger}}), \]
We can assume that $Z$ is small and endowed with an overconvergent presentation $\varprojlim Z_h$. We have to prove that the map
\[ \varinjlim_h H^i(Z_{h}, L\eta_t R\nu'_* \mathbb{M}_{Z_h} ) \to \varinjlim_h \varinjlim_{h'} H^i(Z_{h} \times \mathbf{D}_{h',C}, L\eta_t R\nu'_* \mathbb{M}_{Z_h \times \mathbf{D}_{h',C}} ) \]
is an isomorphism. It suffices to prove it before taking the direct limit over $h$. But this is precisely what Proposition \ref{contractile} tells us.
\end{proof}

Let $\mathcal{F}_{\mathbb{M}}^{\rm mot}$ be the motive of $\mathcal{F}_{\mathbb{M}}$ viewed as a complex of presheaves. As the localization over $\mathcal{S}_{\mathrm{\acute{e}t}}$ of the category of complex of presheaves on $\mathrm{AffSm}^{\dagger}$ is equivalent to the localization over $\mathcal{S}_{\mathrm{\acute{e}t}}$ of the category of complex of presheaves on $\mathrm{RigSm}^{\dagger}$, $\mathcal{F}_{\mathbb{M}}^{\rm mot}$ can and will be seen as an object in $\mathrm{RigMot}$. Proposition \ref{factormotives} implies that for any smooth affinoid dagger variety $Z$,
\begin{eqnarray} \mathrm{Hom}_{\mathrm{RigMot}^{\dagger}}(Z, \mathcal{F}_{\mathbb{M}}^{\rm mot}) = \underset{h} \varinjlim ~ R\Gamma_{\mathrm{\acute{e}t}}(Z_h, L\eta_t R\nu'_* \mathbb{M}), \label{memechose} \end{eqnarray}
for any presentation $\varprojlim Z_h$ of the dagger structure $Z$ on $\hat{Z}$.

\begin{definition}
Let $\mathcal{C} \in \mathrm{RigMot}$. We define, for $\mathbb{M}$ as before and $i\geq 0$,
\[ R\Gamma(\mathcal{C},\mathbb{M})^{\dagger} := \mathrm{Hom}_{\mathrm{RigMot}^{\dagger}}(R\ell_* \mathcal{C} , \mathcal{F}_{\mathbb{M}}^{\rm mot}). \]
To simplify the notation, if $Z \in \mathrm{RigSm}$, $R\Gamma(\qp(Z),\mathbb{M})^{\dagger}$ will simply be denoted by $R\Gamma(Z,\mathbb{M})^{\dagger}$.
\end{definition}

\begin{remarque}
Any complex of presheaves on the category $\mathrm{RigSm}^{\dagger}$ defines a motive. But if one doesn't know that it is a complex of \'etale sheaves and that the \'etale hypercohomology of the complex on $\mathbf{D}_C^{\dagger}$ is trivial, it is hard to describe the functor represented by this motive in terms of the original complex, as in \eqref{memechose}. In particular, it would be impossible to prove the following proposition, which shows that the above definition is reasonable : it is an extension to all smooth quasi-compact rigid spaces (and even to motives) of the usual cohomology of the complex $L\eta_t R\nu'_* \mathbb{M}$ on smooth \textit{proper} varieties.
\end{remarque}

\begin{proposition} \label{propre\'egal}
If $Z$ is a smooth proper rigid variety over $C$, 
\[ R\Gamma(Z,\mathbb{M})^{\dagger} \simeq R\Gamma_{\mathrm{\acute{e}t}}(Z, L\eta_t R\nu'_* \mathbb{M}). \]
\end{proposition}
\begin{proof}
Indeed, one has
\[ R\Gamma(Z, \mathbb{M})^{\dagger} = R\mathrm{Hom}_{\mathrm{RigMot}^{\dagger}}(R\ell_* Z , \mathcal{F}_{\mathbb{M}}) = R\mathrm{Hom}_{\mathrm{RigMot}^{\dagger}}(Z^{\dagger}, \mathcal{F}_{\mathbb{M}}) \]
by Theorem \ref{vezza} (or rather its proof, which tells us that the motive $R\ell_* Z$ is isomorphic to the motive $Z^{\dagger}$, for any choice of a dagger structure $Z^{\dagger}$ on $Z$). Hence we have to show that :
\[  R\mathrm{Hom}_{\mathrm{RigMot}^{\dagger}}(Z^{\dagger}, \mathcal{F}_{\mathbb{M}}) = R\Gamma_{\mathrm{\acute{e}t}}(Z, L\eta_t R\nu'_* \mathbb{M}). \]
To show this, note that as $Z$ is proper, we can fin two finite affinoid coverings $U=(U_i)_{i \in I}$ and $V=(V_i)_{i\in I}$, such that for all $i \in I$, $U_i \Subset V_i$. We can assume that the covering $U$ comes from a covering $U^{\dagger} = (U_i^{\dagger})_{i\in I}$ of $Z^{\dagger}$. We have a Cech spectral sequence
\[ E_1^{p,q}(U) := \bigoplus_{|J|=p+1} H_{\mathrm{\acute{e}t}}^{q}(U_{J}, L\eta_t R\nu'_* \mathbb{M}) \Longrightarrow H_{\mathrm{\acute{e}t}}^{p+q}(Z, L\eta_t R\nu'_* \mathbb{M}), \]
and similarly for the covering $V$. For each $i$, we can choose a presentation $\varprojlim U_{i,h}$ of the dagger structure $U_i^{\dagger}$ on $U_i$ such that $U_{i,h} \Subset V_i$ for all $h$. We have maps of spectral sequences $E_1^{p,q}(U_{h}) \to E_1^{p,q}(U_{h+1}) \to E_1^{p,q}(U)$ for all $h$, which induce isomorphisms on the abutments. But
\[ \underset{h} \varinjlim ~ E_1^{p,q}(U_h) \Longrightarrow \mathrm{Hom}_{\mathrm{RigMot}^{\dagger}}(Z^{\dagger}, \mathcal{F}_{M}^{\rm mot}), \]
because by \eqref{memechose},
\[ \underset{h} \varinjlim ~ E_1^{p,q}(U_h) = \bigoplus_{|J|=p+1} \mathrm{Hom}_{\mathrm{RigMot}^{\dagger}}(U_J^{\dagger}, \mathcal{F}_{\mathbb{M}}^{\rm mot}). \]
Whence the desired isomorphism.
\end{proof}

\section{Construction of the cohomology theory $\mathcal{FF}$} \label{Construction}

In all this section, we will define the cohomology theory $\mathcal{FF}$ with values in $D^b(\mathrm{Coh}_{X_{C^{\flat}}})$ (Definition \ref{defff}) and show that its cohomology groups are vector bundles when $C$ is the completion of an algebraic closure of $W(k)[1/p]$ (Theorem \ref{vb}). We fix once and for all a smooth quasi-compact rigid space $Z$ over $K$\footnote{As the notation may be confusing, let us recall and insist on the fact that $K$ is not, unless explicitly mentioned, assumed to be discretely valued ; for example, we can take $K=C$.} of dimension $n$.

Let us recall that the algebra $B$ is a Fr\'echet-Stein algebra. A coherent sheaf on $X_{C^{\flat}}$ is thus the same thing as a family $(M_I)_I$ of finite type $B_I$-modules, for any compact interval $I=[a,b]$ with rational ends and $a \leq b/p$, endowed with an isomorphism
\[ \varphi^* M_I \otimes_{B_{[a/p,b/p]}} B_{[a,b/p]} \simeq M_I \otimes_{B_I} B_{[a,b/p]}, \]
and such that there are isomorphisms $M_{I'} \otimes_{B_{I'}} B_I \simeq M_I$ whenever $I \subset I'$, satisfying the obvious cocycle condition. To such a coherent sheaf $(M_I)_I$, one associates the $\varphi$-module over $B$
\[ \Gamma(M_I)_I) = \underset{I} \varprojlim ~ M_I. \]
This global sections functor induces an equivalence between the category of coherent sheaves on $Y_{C^{\flat}}$ and its image, which is by definition the category of \textit{coadmissible $\varphi$-modules over $B$}. This category is abelian and stable by extensions.  

\begin{proposition} \label{B_I}
Let $i \geq 0$. The $\varphi$-module $H^i(Z_C,\mathbb{B})^{\dagger}$ over $B$ is a coadmissible $\varphi$-module over $B$.
\end{proposition}
\begin{proof}
Set, for any compact interval $I$, $M_I=H^i(Z_C,\mathbb{B}_I)^{\dagger}$. We claim that the family $(M_I)_I$ defines a coherent sheaf on $X_{C^{\flat}}$ and that :
\[ H^i(Z_C,\mathbb{B})^{\dagger} = \underset{I} \varprojlim ~ M_I. \]
To check that it defines a coherent sheaf, we can reduce to the case where $Z$ is proper smooth, by Theorem \ref{ayoub}, as all the objects involved factor through the category of motives. Again, we apply Lemma \ref{induction} and the fact that $\mathbb{B}_I/t$ is isomorphic to a finite number of copies of $\widehat{\O}$ indexed by the zeroes of $t$ in the annulus of radius $I$. As $\varphi(t)=pt$ and $p$ is invertible, the only non trivial thing to check is that the family $(N_I)_I$, with $N_I=H_{\mathrm{pro\acute{e}t}}^i(Z_C,\mathbb{B}_I)$ defines a coherent sheaf on $X_{C^{\flat}}$. This is a direct consequence of Proposition \ref{proprepro\'et}. (Another argument for the finiteness of $H_{\mathrm{pro\acute{e}t}}^i(Z_C, \mathbb{B})$, which does not rely on Faltings' comparison result, is given in \cite[Th. 8.1]{KL} : this roughly amounts to do a \og Cartan-Serre argument \fg{} directly for $\mathbb{B}$, in the spirit of the one of \cite{SHodge} for the constant sheaf $\qp$.) 

Finally, by Lemma \ref{44}, we have a short exact sequence
\[ 0 \to R^1 \underset{I} \varprojlim ~ H^{i-1}(Z_C,\mathbb{B}_I)^{\dagger} \to H^{i}(Z_C,\mathbb{B})^{\dagger} \to   \underset{I} \varprojlim ~ H^{i}(Z_C,\mathbb{B}_I)^{\dagger} \to 0. \]
But we have just seen that $H^{i-1}(Z_C,\mathbb{B}_{I'})^{\dagger} \otimes_{B_{I'}} B_I \simeq H^{i-1}(Z_C,\mathbb{B}_I)^{\dagger}$ whenever $I \subset I'$. Therefore the derived inverse limit on the left vanishes, and thus we indeed have
 \[ H^i(Z_C,\mathbb{B})^{\dagger} = \underset{I} \varprojlim ~ M_I, \]
as desired. 
\end{proof}

\begin{remarque}
When $Z$ has dimension $1$, one can give another argument which does not appeal to Theorem \ref{ayoub}. One can assume $Z$ to be affinoid. Then one can always write (\cite{vdp})
\[ Z_C = Z' ~ \backslash ~ \left(\bigcup_{k=1}^r D_k(r_k) \right), \]
where $Z'$ is a smooth proper curve and the $D_k(r_k)$ are open disks of radius $r_k$ (the circle of same origin and radius as $D_k(r_k)$ will be called $C_k$). Let for $n$ big enough :
\[ Z_n = Z' ~ \backslash ~ \left(\bigcup_{k=1}^r D_k(r_k-1/n) \right), \]
We get from this incision-excision triangles 
\[ \underset{n} \varinjlim ~ R\Gamma_{Z_n}(Z', L\eta_t R\nu'_* \mathbb{B}) \to  R\Gamma(Z_C,\mathbb{B})^{\dagger} \to \bigoplus_{k=1}^r R\Gamma(C_k, \mathbb{B})^{\dagger} \overset{+1} \longrightarrow  \]
and 
\[ \underset{n} \varinjlim ~ R\Gamma_{Z_n}(Z', L\eta_t R\nu'_* \mathbb{B}) \to R\Gamma(Z', L\eta_t R\nu'_* \mathbb{B}) \to \bigoplus_{k=1}^r R\Gamma(D_k(r_k), L\eta_t R\nu'_* \mathbb{B}) \overset{+1} \longrightarrow  \]
(see the proof of \cite[Th. 2.8.28]{theseaclb}). So to show that the cohomology groups of $R\Gamma(Z_C,\mathbb{B})^{\dagger}$ are coadmissible, we are reduced to prove that the same is true for the cohomology groups of $R\Gamma(Z', L\eta_t R\nu'_* \mathbb{B})$, $R\Gamma(D_k(r_k), L\eta_t R\nu'_* \mathbb{B})$ and $R\Gamma(C_k, \mathbb{B})^{\dagger}$: for the first ones, we can argue as in the proof of the proposition ; the second ones were already computed (in disguise) in Proposition \ref{contractile} (cf. also \cite[Cor. 2.3.30]{theseaclb}) ; for the last ones, see \cite[Cor 2.3.29]{theseaclb}. 
\end{remarque}

\begin{definition} \label{defff}
Let $\mathcal{FF}(Z) \in D^b(\mathrm{Coh}_{X_{C^{\flat}}})$ be the complex of coherent sheaves on $X_{C^{\flat}}$ corresponding to the complex of $\varphi$-equivariant coadmissible modules $R\Gamma(Z_C,\mathbb{B})^{\dagger}$ over $B$ (we used the fact that on a noetherian scheme $S$, the bounded derived category of coherent sheaves identifies with the full subcategory of the bounded derived category of $\O_S$-modules formed by complexes whose cohomology sheaves are coherent, plus GAGA theorem for the Fargues-Fontaine curve).

For any $i\geq 0$, let
\[ H_{\mathcal{FF}}^i(Z) = H^i(\mathcal{FF}(Z)), \]
which is a coherent sheaf on $X_{C^{\flat}}$.
\end{definition}

\begin{remarque}
For any $i\geq 0$, the coherent sheaf $H_{\mathcal{FF}}^i(Z)$ is obviously $\mathcal{G}_K$-equivariant\footnote{Any element of $\mathrm{Aut}(C)$ acts on $X_{C^{\flat}}$ ; in particular $\mathcal{G}_K$ acts.}. All the identifications made below will be compatible with the $\mathcal{G}_K$-action, even if we will not explicitely state it. This equivariant structure will be only used in section \ref{relation}, when $K$ is discretely valued and $Z$ proper and smooth.
\end{remarque}  

\begin{remarque}
If one is just interested in defining individual cohomology groups instead of a complex, one can define $H_{\mathcal{FF}}^i(Z)$ to be the coherent sheaf on $X_{C^{\flat}}$ corresponding to the pair $(H^i(Z_C, \mathbb{B}_e)^{\dagger},H^i(Z_C,\mathbb{B}_{\rm dR}^+)^{\dagger})$. One of course has to check that the natural map
\[ H^i(Z_C, \mathbb{B}_e)^{\dagger} \otimes_{B_e} B_{\rm dR} \to H^i(Z_C,\mathbb{B}_{\rm dR})^{\dagger} \]
is a quasi-isomorphism but as both sides factor through the category of motives $\mathrm{RigMot}$, using Theorem \ref{ayoub} once again, it is enough to show that the above map is an isomorphism for $Z$ proper and smooth. This is a direct consequence of Proposition \ref{proprepro\'et}.

The next proposition will show that this gives the same groups, and it is this description that we will need to prove that we always get vector bundles when $C$ is the completion of an algebraic closure of $W(k)[1/p]$.
\end{remarque}


\begin{proposition} \label{modeleformel}
The natural map
\[ R\Gamma(Z_C,\mathbb{B})^{\dagger} \otimes_B B_{\rm dR}^+ \simeq R\Gamma(Z_C, \mathbb{B}_{\rm dR}^+)^{\dagger} \]
is a quasi-isomorphism.
\end{proposition}
\begin{proof}
This is proved exactly as in the proof of Proposition \ref{B_I}.
\end{proof}

\begin{remarque}
This implies in particular, by Proposition \ref{B_I}, that $H^i(Z_C, \mathbb{B}_{\rm dR}^+)^{\dagger}$ is a $B_{\rm dR}^+$-module of finite type. This can be seen directly, as follows : by Theorem \ref{ayoub}, it suffices to treat the case where $Z$ is smooth and proper, in which case we can invoke \cite[Th. 13.8]{BMS}. Proposition \ref{bplusdr} below gives another argument, when $C$ is the completion of an algebraic closure of $W(k)$.
\end{remarque}

\begin{lemme} \label{chap13}
Assume that $C$ is the completion of an algebraic closure of $W(k)[1/p]$, and let $F$ be a finite extension of $W(k)[1/p]$. Let $\mathcal{C}$ be a compact rigid motive over $F$. One has for all $i$ an isomorphism 
\[ H^i(\mathcal{C}_C,\mathbb{B}_{\rm dR}^+)^{\dagger} \simeq H_{\rm dR}^i(\mathcal{C}/F)^{\dagger} \otimes_F B_{\rm dR}^+, \]
where $H_{\rm dR}^i(\mathcal{C}/F)^{\dagger}$ is the overconvergent de Rham cohomology of $\mathcal{C}$, as defined in \cite{Vezzani}. 
\end{lemme}
\begin{proof}
Let us show that for any compact motive $\mathcal{C}$ defined over $F$, one has a quasi-isomorphism :
\[ R\Gamma(\mathcal{C}_C, \mathbb{B}_{\rm dR}^+)^{\dagger} \simeq R\Gamma_{\rm dR}(\mathcal{C}/F)^{\dagger} \otimes_F B_{\rm dR}^+. \]
It suffices to construct explicit complexes computing both sides and a quasi-isomorphism between them, for any smooth affinoid rigid space $T= \mathrm{Spa}(R,R^{\circ})$ over $F$, \textit{functorially} in $R$ (they will then glue). For this, let $(T_h)_h$ be a presentation of a dagger structure on $T$. Recall that Elkik's theorem (\cite{elkik}) gives for any $h$ an equivalence
\[ \underset{F' | F} \varprojlim \widetilde{T}_{h,F',\mathrm{\acute{e}t}} \simeq \widetilde{T}_{h,C,\mathrm{\acute{e}t}}, \]
where $F'$ runs among the finite extensions of $F$. If $\mathcal{F}$ is a vector bundle on $T_h$ and $F'$ a finite extension of $F$, $\mathcal{F} \hat{\otimes}_F B_{\rm dR}^+$ is a sheaf of $\O_{T_{h,F'}}^{\dagger}$-modules, as $B_{\rm dR}^+$ is a $\bar{F}$-algebra. In other words, $\mathcal{F} \hat{\otimes}_F B_{\rm dR}^+$ defines an object of the left hand side of the above equivalence, and can thus be seen as an \'etale sheaf on $T_{h,C}$, to which we'll give the same name. One shows using Poincar\'e's lemma that for any $h$, the natural map
\[ \mathbb{B}_{\mathrm{dR},T_{h}}^+ \to \O\mathbb{B}_{\mathrm{dR},T_{h}}^+ \otimes_{\O_{T_{h}}} \Omega_{T_{h}}^{\bullet} \]
induces an isomorphism
\[ R\nu'_* \mathbb{B}_{\mathrm{dR},T_{h}}^+ \simeq \O_{T_h} \hat{\otimes}_F B_{\rm dR}^+ \to \Omega_{T_h}^1 \hat{\otimes}_F t^{-1} B_{\rm dR}^+ \to \dots \to \Omega_{T_h}^n \hat{\otimes}_F t^{-n} B_{dR}^+ \]
where $n=\dim T_h$ and the differential of the complex on the right hand side is just the usual differential of the de Rham complex (the terms of the complex on the right being seen as \'etale sheaves on $T_{h,C}$). Applying $L\eta_t$ on both sides, we get a quasi-isomorphism :
\[ L\eta_t R\nu'_* \mathbb{B}_{\mathrm{dR},T_{h}}^+ \simeq \Omega_{T_h}^{\bullet} \hat{\otimes}_F B_{\rm dR}^+. \]
For the details, see \cite[Prop. 2.3.16, Rem. 2.3.18]{theseaclb}. Then apply $R\Gamma(T_{h,C},\cdot)$ and take the direct limit over $h$. We get
\[ \varinjlim ~ R\Gamma(T_{h,C},L\eta_t R\nu'_* \mathbb{B}_{\mathrm{dR},T_{h}}^+) \simeq \underset{h} \varinjlim ~ R\Gamma(T_h, \Omega_{T_h}^{\bullet} \hat{\otimes}_F B_{\rm dR}^+) \simeq \underset{h} \varinjlim ~ (\Omega^{\bullet}(T_h) \hat{\otimes}_F B_{\rm dR}^+ ) \]
(the last equality by quasi-compacity). Thus the specific complex on the right computes the left hand side, which is isomorphic to $R\Gamma(T_C, \mathbb{B}_{\rm dR}^+)^{\dagger}$. The inductive systems $(\Omega^{\bullet}(T_h) \hat{\otimes}_F B_{\rm dR}^+)_h$ and $(\Omega^{\dagger,\bullet}(T_h) \hat{\otimes}_F B_{\rm dR}^+)_h$ are coinitial. Hence we can replace in the inductive limit the de Rham complexes by overconvergent de Rham complexes, i.e. we get a quasi-isomorphism :
\[ R\Gamma(T_C, \mathbb{B}_{\rm dR}^+)^{\dagger} \simeq \underset{h} \varinjlim ~ ( \Omega^{\dagger,\bullet}(T_h) \hat{\otimes}_F B_{\rm dR}^+ ). \]
The natural map, for all $h$ :
\[  \Omega^{\dagger,\bullet}(T_h) \otimes_F B_{\rm dR}^+ \to \Omega^{\dagger,\bullet}(T_h) \hat{\otimes}_F B_{\rm dR}^+ \]
is a quasi-isomorphism. Indeed, by \cite[Lem. 4.7]{grosse} and Lemma \ref{schneider} below, we have for all $k$
\[ H^k \left( \Omega^{\dagger,\bullet}(T_h) \hat{\otimes}_F B_{\rm dR}^+ \right) = H_{\rm dR}^k(T_h)^{\dagger} \hat{\otimes}_F B_{\rm dR}^+. \]
and we can get rid of the completed tensor product, as we know that $H_{\rm dR}^k(T_h/F)^{\dagger}$ is a finite dimensional $F$-vector space for all $h$ (\cite[Cor. 5.21]{Vezzani}). So, we obtain a quasi-isomorphism :
\[ \underset{h} \varinjlim ~ \Omega^{\bullet}(T_h) \otimes_F B_{\rm dR}^+ \to \underset{h} \varinjlim ~ (\Omega^{\bullet}(T_h) \hat{\otimes}_F B_{\rm dR}^+ ). \]
We then take the inverse limit of these isomorphisms over all possible choices of dagger structures on $T$ to make everything functorial in $R$. . 
\end{proof}

The following lemma was used in the proof.

\begin{lemme} \label{schneider}
Let $K^{\bullet}$ be a strict complex of $\qp$-Fr\'echet spaces, let $W$ be a $\qp$-Banach space. Then one has for all $k$
\[ H^k( K^{\bullet} \hat{\otimes}_{\qp} W) = H^k(K^{\bullet}) \hat{\otimes}_{\qp} W. \]
\end{lemme}

\begin{proposition} \label{bplusdr}
Assume that $C$ is the completion of an algebraic closure of $W(k)[1/p]$. Let $i\geq 0$. The $B_{\rm dR}^+$-module $H^i(Z_C,\mathbb{B}_{\rm dR}^+)^{\dagger}$ is finite free. 
\end{proposition}
\begin{proof}
It is a direct consequence of Proposition \ref{2lim}, the last lemma and \cite[Cor 5.21]{Vezzani}.
\end{proof}

\begin{remarque}
If we don't make any assumption of $C$, it is still true that there exists a smooth affinoid rigid space $S=\mathrm{Spa}(A,A^{\circ})$ over $W(k)[1/p]$, with a map $A \to C$, such that the motive $\qp(Z_C)$ descends to a motive over $S$. If $\mathcal{C}$ is a rigid motive over $S$, Vezzani defined in \cite{Vezzani} the overconvergent relative de Rham cohomology $H_{\rm dR}^*(\mathcal{C}/S)^{\dagger}$ of $\mathcal{C}$ over $S$. If we knew that these groups were $A$-modules of finite type when $\mathcal{C}$ is quasi-compact (i.e. if we had a relative version of the finiteness theorem for overconvergent de Rham cohomology), we could prove as before that one has an isomorphism :
\[ H^i(\mathcal{C}_C,\mathbb{B}_{\rm dR}^+)^{\dagger} \simeq H_{\rm dR}^i(\mathcal{C}/S)^{\dagger} \otimes_A B_{\rm dR}^+ \]
for any quasi-compact motive $\mathcal{C}$ over $S$. As relative de Rham cohomology groups, the groups $H_{\rm dR}^k(\mathcal{C}/S)^{\dagger}$ are endowed with an integrable connection, the Gauss-Manin connection (as defined in \cite{katzoda}). But any coherent sheaf with an integrable connection on a space over a characteristic $0$ field is automatically locally free, so this would give the result in this case too.
\end{remarque} 

\begin{remarque} \label{derhamlisse}
Lemma \ref{homotopie} and the motivic definition of the overconvergent de Rham cohomology have another interesting application, unrelated to the rest of this paper. Let $G$ be a locally profinite group acting continuously (in the sense of \cite[Def. 2.1, Lem. 2.2]{padicLT}) on a smooth Stein variety $Z$, of dimension $d$. Then the action of $G$ on $H_{\mathrm{dR},c}^i(Z)$ is smooth, for any $i\geq 0$. Indeed let $(U_j)$ be a Stein cover (that is, an increasing covering by affinoid $U_j$ such that the transition maps $\O(U_{j+1}) \to \O(U_j)$ are compact with dense image for all $j$) of $Z$. Let $v \in H_{\mathrm{dR},c}^i(Z)$. As
\[ H_{\mathrm{dR},c}^i(Z) = \underset{j} \varinjlim ~ H_{\mathrm{dR},c}^i(U_j)^{\dagger}, \]
there exists $j$ such that $v \in H_{\mathrm{dR},c}^i(U_j)^{\dagger}$. One can find a compact open subgroup $H$ of $G$ stabilizing $U_j$. It is thus enough to show that the action of $H$ on $H_{\mathrm{dR},c}^i(U_j)^{\dagger}$ is smooth. As this space is finite dimensional, this amounts to say that the action of $H$ factors through a finite quotient and this can be checked on the dual, which, by Poincar\'e duality (\cite{grosse}), identifies with $H_{\rm dR}^{2d-i}(U_j)^{\dagger}$. But if $g \in H$ is close enough to the identity, the morphism it induces on $U_j$ is homotopic to the identity by Lemma \ref{homotopie}, thus equal to the identity on the associated motive. As $H_{\rm dR}^{2d-i}(U_j)^{\dagger}$ depends only on the motive attached to $U_j$, $g$ acts trivially on it.

This argument also applies to $\ell$-adic cohomology with compact supports ($\ell \neq p$), where the result was already proved by Berkovich \cite[\S 7]{berko}. The smoothness of the action on de Rham cohomology with compact supports was already checked by hand in \cite{DLB}, in the particular case where $G=GL_2(\qp)$ and $Z$ is a Drinfeld covering of the non-archimedean half plane. 
\end{remarque}

\begin{corollaire} \label{bidon}
Assume that $C$ is the completion of an algebraic closure of $W(k)[1/p]$. Let $i \geq 0$. The $B_{\rm dR}^+$-module $H^i(Z_C,\mathbb{B})^{\dagger} \otimes_B B_{\rm dR}^+$ is finite free.
\end{corollaire}
\begin{proof}
Combine Propositions \ref{modeleformel} and \ref{bplusdr}.
\end{proof}

\begin{proposition} \label{B_e}
Assume that $C$ is the completion of the algebraic closure of $W(k)[1/p]$. Let $i \geq 0$. The $B_e$-module $H^0(X_{C^{\flat}} \backslash \{ \infty \}, H_{\mathcal{FF}}^i(Z))$ is a finite free $B_e$-module.
\end{proposition}
\begin{proof}
By Proposition \ref{2lim}, there exists a finite extension $F$ of $W(k)[1/p]$ such that the motive of $Z$ is defined over $F$. As $B_e$ is principal, to check that $H^0(X_{C^{\flat}} \backslash \{ \infty \}, H_{\mathcal{FF}}^i(Z))$ is finite free over $B_e$ amounts to say that it is torsion free, as we already know that it is of finite type. But if $H_{\mathcal{FF}}^i(Z)_{|_{X_{C^{\flat}} \backslash \{\infty \}}}$ had a torsion part, its support would be a finite set of points of $X_{C^{\flat}} \backslash \{ \infty\}$, stable by $\mathcal{G}_F$ and such a set is necessarily empty, by \cite[Prop. 10.1.1]{FF}. 
\end{proof}

\begin{theoreme} \label{vb}
Assume that $C$ is the completion of the algebraic closure of $W(k)[1/p]$. Let $i \geq 0$. The sheaf $H_{\mathcal{FF}}^i(Z)$ is a vector bundle on $X_{C^{\flat}}$. It vanishes for $i<0$ and $i>2n$. 
\end{theoreme}
\begin{proof}
Indeed, Corollary \ref{bidon} tells that $\widehat{H_{\mathcal{FF}}^i(Z)_{\infty}}$ is a finite free $B_{\rm dR}^+$-module and Proposition \ref{B_e} that $H^0(X_{C^{\flat}} \backslash \{ \infty \}, H_{\mathcal{FF}}^i(Z))$ is a finite free $B_e$-module.

The last sentence is a direct corollary of the corresponding assertion for overconvergent de Rham cohomology.
\end{proof}

\begin{remarque} \label{isocristal}
Assume that $C$ is the completion of an algebraic closure of $W(k)[1/p]$. Let $i\geq 0$. Then for any smooth quasi-compact rigid space $Z$ defined over $K$, one gets an isocrystal $H_{\mathrm{isoc}}^i(Z)$. Indeed, let $\mathcal{E}$ be the functor associating to an isocrystal $D$ the vector bundle $\mathcal{E}(D)$ on $X_{C^{\flat}}$ whose geometric realization is given by $Y_{C^{\flat}} \times_{\varphi} D$ ; let $\mathrm{gr}$ be the functor with values in the category of $\q$-graded vector bundles on $X_{C^{\flat}}$, sending a vector bundle $\mathcal{E}$ on $X_{C^{\flat}}$ to the direct sum of the graded parts of the grading given by its Harder-Narasimhan filtration. The composite functor $\mathrm{gr} \circ \mathcal{E}$ induces an equivalence between the category $\varphi-\mathrm{Mod}_{\breve{\q}_p}$ and the category of $\q$-graded vector bundles on $X_{C^{\flat}}$ such that for each $\lambda \in \q$, the $\lambda$-graded part is semi-stable of slope $\lambda$ (\cite[Lem. 3.6]{anschuetz}). Applying this to the direct sum of the graded parts of $H_{\mathcal{FF}}^i(Z)$, with the grading induced by its Harder-Narasimhan filtration, we obtain the isocrystal $H_{\rm isoc}^i(Z)$. 

If one moreover assumes that $Z_C$ has a proper smooth formal model $\mathfrak{Z}$ over $\O_{C}$, $H_{\mathrm{isoc}}^i(Z)$ is the isocrystal underlying $H_{\rm crys}^i(\mathfrak{Z}_s)$\footnote{We are implicitly using here the fact that the categories of isocrystals over $\overline{\mathbf{F}_p}$ and $k$ are the same.}, for all $i\geq 0$, with $\mathfrak{Z}_s = \mathfrak{Z} \times_{\O_{K}} k$, as we will see in the next section.

Unfortunately a quasi-inverse of the above equivalence is not necessarily exact, so we don't get a well-behaved cohomology with values in the category of isocrystals that way. Such a theory exists so far only at the level of formal schemes (in accordance with Scholze's cohomological picture, \cite[\S 8]{SICM}).  
\end{remarque}

\begin{remarque}
Let $K$ be a discretely valued extension of $\qp$ with perfect residue field and $i\geq 0$. Let $Z$ be a smooth quasi-compact rigid space over $K$. By Proposition \ref{B_e}, if we set 
\[ H_{\mathrm{isoc},+}^i(Z) := \bigcup_{[K':K] < \infty} (H^0(X \backslash \{ \infty \}, H_{\mathcal{FF}}^i(Z)) \otimes_{B_e} B_{\rm st})^{\mathcal{G}_{K'}}. \]
this defines a $(\varphi,N,\mathcal{G}_K)$-module over $W(k)[1/p]$, such that if $Z$ has a proper semi-stable formal model $\mathfrak{Z}$, 
\[ H_{\mathrm{isoc},+}^i(Z) \simeq H_{\rm logrig}^i(\mathfrak{Z}_s) \otimes_{K_0} W(k)[1/p], \]
for all $i\geq 0$, by the classical comparison theorems in $p$-adic Hodge theory.

The isocrystal underlying $H_{\mathrm{isoc}, +}^i(Z)$ is $H_{\rm isoc}^i(Z)$. The recipe of Remark \ref{isocristal} does not allow to recover $H_{\mathrm{isoc},+}^i(Z)$ with its additional structures, but it has the advantage that it does not use the Galois action and thus works over any $C$ as in the corollary.
\end{remarque}

\section{Relation with other cohomology theories} \label{relation}

In this section, we give proofs of the assertions of Remarks \ref{propercase} and \ref{semistablecase}, and discuss briefly the relation of the d\'ecalage functor $L\eta_t$ with the pro-\'etale cohomology of $\qp$ and syntomic cohomology on rigid spaces.
\\

\textit{Proof of Remark \ref{propercase}.} Let $K$ be a discretely valued extension of $\qp$ with perfect residue field $\kappa$ and $V$ be a $p$-adic representation of $\mathcal{G}_K$, which is semi-stable. Let
\[ D_{\rm st}(V)= (V \otimes_{\qp} B_{\rm st})^{\mathcal{G}_K}. \]
This a $(\varphi,N)$-module over $K$. To such a $(\varphi,N)$-module $D$, Fargues and Fontaine \cite[Ch. 10]{FF} attach a $\mathcal{G}_K$-equivariant vector bundle $\mathcal{E}(D)$ on $X_{C^{\flat}}$ corresponding to the $\mathcal{G}_K$-equivariant $B$-pair
\[ ((D \otimes_{K_0} B_{\rm st})^{\varphi=1,N=0}, D \otimes_{K_0} B_{\rm dR}^+) \]
($K_0$ is the maximal unramified extension of $W(\kappa)[1/p]$ contained in $K$). If $V$ is only assumed to be de Rham, there exists by Berger's theorem a finite Galois extension $L$ of $K$ such that $V_{|_{\mathcal{G}_L}}$ is semi-stable : it gives rise to a $\mathcal{G}_L$-equivariant vector bundle on $X_{C^{\flat}}$ with a descent datum from $L$ to $K$ (\cite[Def. 10.6.3]{FF}), that is to a $\mathcal{G}_K$-equivariant vector bundle on $X_{C^{\flat}}$ (\cite[Prop 10.6.4]{FF}). This vector bundle corresponds to the $\mathcal{G}_K$-equivariant $B$-pair :
\[ ((D_{\rm pst}(V) \otimes_{W(k)[1/p]} B_{\rm st})^{\varphi=1,N=0}, D_{\rm pst}(V) \otimes_{W(k)[1/p]} B_{\rm dR}^+), \]
where
\[ D_{\rm pst}(V) = \bigcup_{[K':K]<\infty} D_{\rm st}(V_{|_{\mathcal{G}_{K'}}}). \]

Let $Z$ be a proper smooth rigid space defined over $K$, and $i\geq 0$. We know by \cite{SHodge} that $H_{\mathrm{\acute{e}t}}^i(Z_{C}, \qp)$ is a de Rham representation.

\begin{proposition} 
The $\mathcal{G}_K$-equivariant vector bundle on $X_{C^{\flat}}$ attached to $H_{\mathrm{\acute{e}t}}^i(Z_{C}, \qp)$ by the above recipe is isomorphic to the $\mathcal{G}_K$-equivariant vector bundle $H_{\mathcal{FF}}^i(Z)$. 
\end{proposition}
\begin{proof}
We will check that $H_{\mathcal{FF}}^i(Z)$ corresponds to the same $\mathcal{G}_K$-equivariant $B$-pair. By Proposition \ref{propre\'egal}, we have
\[ R\Gamma(Z_{C}, \mathbb{B})^{\dagger} = R\Gamma_{\mathrm{\acute{e}t}}(Z_{\cp}, L\eta_t R\nu'_* \mathbb{B}). \]
After inverting $t$, this is just
\[ R\Gamma_{\mathrm{pro\acute{e}t}}(Z_{C},  \mathbb{B}[1/t]) = R\Gamma_{\mathrm{\acute{e}t}}(Z_{C}, \qp) \otimes_{\qp} B[1/t]. \]
Around $\infty$, it's given by $R\Gamma_{\mathrm{\acute{e}t}}(Z_{C}, L\eta_t R\nu'_* \mathbb{B}_{\rm dR}^+)$. Using \ref{chap13} or \cite[Ch. 13]{BMS},
\[ H^i(R\Gamma_{\mathrm{\acute{e}t}}(Z_{C}, L\eta_t R\nu'_* \mathbb{B}_{\rm dR}^+)) =   H_{\rm dR}^i(Z) \otimes_K B_{\rm dR}^+. \]
In other words, $H_{\mathcal{FF}}^i(Z)$ corresponds to the $B$-pair
\[ (H_{\mathrm{\acute{e}t}}^i(Z_{C}, \qp) \otimes_{\qp} B_e, H_{\rm dR}^i(Z) \otimes_K B_{\rm dR}^+). \] 
The comparison theorem gives an isomorphism
\[ H_{\mathrm{\acute{e}t}}^i(Z_{C}, \qp) \otimes_{\qp} B_{\rm st} = D_{\rm pst}(H_{\mathrm{\acute{e}t}}^i(Z_{C}, \qp)) \otimes_{W(k)[1/p]} B_{\rm st} \]
compatible with $\varphi$, $N$ and the $\mathcal{G}_K$-action on both sides. Taking $\varphi=1$, $N=0$ shows that the two $B_e$-modules with $\mathcal{G}_K$-action considered are the same. It is also true for the two $B_{\rm dR}^+$-modules, since
\[ H_{\rm dR}^i(Z) \otimes_K B_{\rm dR}^+ = D_{\rm pst}(H_{\mathrm{\acute{e}t}}^i(Z_{C}, \qp)) \otimes_{W(k)[1/p]} B_{\rm dR}^+. \] 
\end{proof}

\begin{remarque} \label{hyodokato}
When $Z$ has a proper semi-stable formal model $\mathfrak{Z}$ over $\O_K$, the $(\varphi,N)$-module $D_{\rm st}(H_{\mathrm{\acute{e}t}}^i(Z_C,\qp))$ is isomorphic to the Hyodo-Kato (or log-rigid) cohomology $H_{\mathrm{logrig}}^i(\mathfrak{Z}_s)$. 
\end{remarque}

\textit{Proof of Remark \ref{semistablecase}.} We start with a

\begin{conj} \label{compconj}
Let $\mathfrak{Z}$ be a smooth quasi-compact formal scheme over $\O_C$. Let $Z$ be its rigid generic fiber and $\mathfrak{Z}_s=\mathfrak{Z} \times_{\O_C} k$ its special fiber. Then for all $i \geq 0$,
\[  H_{\mathcal{FF}}^i(Z) = \mathcal{E}(H_{\rm rig}^i(\mathfrak{Z}_s)). \]
\end{conj} 

\begin{remarks} a) One can also formulate a similar conjecture in the semi-stable case, using log-rigid cohomology.

b) If true, this conjecture allows to give an alternative definition of rigid cohomology for smooth varieties over $k$, inspired by Vezzani \cite{Vezzani2}. Let $W$ be a smooth scheme over $k$. Fix a section $k \to \O_{C^{\flat}}$ of the canonical projection. Let $Z^{\flat}$\footnote{Beware that $Z^{\flat}$ is not the tilt of anything. It is just a notation.}  be the analytification of the base change of $W$ to $C^{\flat}$ and denote by $Z$ the rigid motive over $C$ attached to $Z^{\flat}$ by the equivalence of \cite{FW} between rigid motives over $C$ and $C^{\flat}$. The for all $i \geq 0$, the isocrystals $H_{\rm rig}^i(W)$ and $H_{\rm isoc}^i(Z)$ (defined in Remark \ref{isocristal}) are isomorphic.
\end{remarks}

We now prove Conjecture \ref{compconj} in the proper case, in which case rigid cohomology is just crystalline cohomology. It is an easy consequence of the results of \cite[Ch. 12, 13]{BMS}.
\begin{proposition} \label{criscomp}
Let $\mathfrak{Z}$ be a smooth proper formal scheme over $\O_C$. Let $Z$ be its rigid generic fiber and $\mathfrak{Z}_s=\mathfrak{Z} \times_{\O_C} k$ its special fiber. Then for all $i \geq 0$,
\[  H_{\mathcal{FF}}^i(Z) = \mathcal{E}(H_{\rm crys}^i(\mathfrak{Z}_s)). \]
\end{proposition} 
\begin{proof}
Let $i \geq 0$ and $I$ be a compact interval contained in $]p^{-1},1[$. By Proposition \ref{propre\'egal}, we know that
\[ H^i(Z,\mathbb{B}_I)^{\dagger} = H_{\mathrm{\acute{e}t}}^i(Z,L\eta_t R\nu'_* \mathbb{B}_I). \]
Applying Proposition \ref{cj}, we get :
\[ H^i(Z,\mathbb{B}_I)^{\dagger} \simeq H_{\mathrm{\acute{e}t}}^i(Z,L\eta_{\mu} R\nu_* \mathbb{A}_{\rm inf}) \otimes_{A_{\rm inf}} B_I \]
(the cohomology groups on the right are finitely presented $A_{\rm inf}$-modules). Our choice of $I$ allows to rewrite the right hand side as :
\[ H_{\mathrm{\acute{e}t}}^i(Z,L\eta_{\mu} R\nu_* \mathbb{A}_{\rm inf}) \otimes_{A_{\rm inf}} B_I = (H_{\mathrm{\acute{e}t}}^i(Z,L\eta_{\mu} R\nu_* \mathbb{A}_{\rm inf}) \otimes_{A_{\rm inf}} B_{\rm crys}^+) \otimes_{B_{\rm crys}^+} B_I. \] 
Now \cite[Th. 12.1]{BMS}, combined with \cite[Prop. 13.9]{BMS}, tells us that
\[ H_{\mathrm{\acute{e}t}}^i(Z,L\eta_{\mu} R\nu_* \mathbb{A}_{\rm inf}) \otimes_{A_{\rm inf}} B_{\rm crys}^+ \simeq H_{\rm crys}^i(\mathfrak{Z}_s) \otimes_{W(k)[1/p]} B_{\rm crys}^+. \]
Hence we have an isomorphism :
\[ H^i(Z,\mathbb{B}_I)^{\dagger} \simeq H_{\rm crys}^i(\mathfrak{Z}_s) \otimes_{W(k)[1/p]} B_I. \]
This isomorphism holds for any compact interval $I \subset ]p^{-1},1[$ and is compatible with Frobenius structures in the obvious sense (see the beginning of Paragraph \ref{Construction}). It therefore implies the isomorphism of the proposition.
\end{proof}

\begin{remarque} 
Using Proposition \ref{cj} and the results of \u{C}esnavi\u{c}ius- Koshikawa (namely, \cite[Th. 5.4]{CJ} and \cite[Prop. 9.2]{CJ}), one can extend the above proof to the semi-stable case.
\end{remarque}

\textit{Geometric syntomic cohomology : \'etale-syntomic comparison theorem.} The reader is referred to \cite[\S 2.8]{theseaclb} for some results regarding the relation between syntomic cohomology and the d\'ecalage functor $L\eta_t$. In this paragraph, we simply explain how to give a short proof of the \'etale-syntomic comparison theorem in the proper good reduction case, which is essentially a reformulation of the proof of Colmez-Nizio\l{} \cite{CN} in the language of Bhatt-Morrow-Scholze \cite{BMS}. We start with an easy generalization of Lemma \ref{induction}. 

\begin{lemme} \label{synto}
We keep the notations of Lemma \ref{induction}. For $r \geq 0$, let $\epsilon_r= r+\delta_r$. For any $r\geq 0$, one has a distinguished triangle
\[ \eta_{\epsilon_r,f} K^{\bullet} \longrightarrow \eta_{f} K^{\bullet} \longrightarrow (\eta_f K^{\bullet}/\mathrm{Fil}^r) \overset{+1} \longrightarrow ~ , \]
the filtration being the tensor product of the filtration b\^ete by the $f$-adic filtration.
\end{lemme}

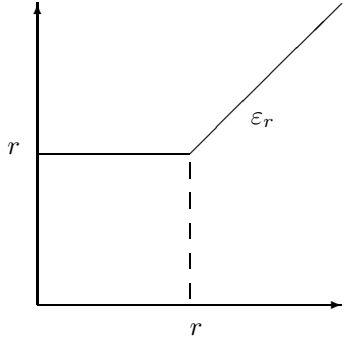
\setlength{\unitlength}{4cm}
\begin{center}
\begin{picture}(1.5,1.3)
\put(0.1,0.1){\vector(0,1){1}}
\put(0.1,0.1){\vector(1,0){1}}
\put(0.1,0.6){\line(1,0){0.5}}
\multiput(0.6,0.12)(0,0.1){5}
{\line(0,1){0.05}}
\put(0.6,0.6){\line(1,1){0.5}}
\put(0,0.6){$r$}
\put(0.6,0){$r$}
\put(0.8,0.7){$\epsilon_r$}
\end{picture}
\end{center}

We will apply the lemma to the complex $R\nu'_* \mathbb{B}$, and $f=t$. Assume that $K$ is a complete discretely valued extension of $\qp$ with perfect residue field. Let $\mathfrak{Z}$ be a smooth proper formal scheme over $\O_K$, with rigid generic fiber $Z$, and $0 \leq i \leq r$. The $\varphi$-module $H^i(Z_C,L\eta_{\epsilon_r,t} R\nu'_*t^r \mathbb{B})$ over $B$ gives rise to a sheaf on $X_{C^{\flat}}$ and we denote by $\mathcal{E}^i(Z,r)$ its twist by $\O_{X_{C^{\flat}}}(r)$. We claim that its global sections are isomorphic to the geometric syntomic cohomology group $H_{\rm syn}^i(\mathfrak{Z}_{\O_C},r)[1/p]$. Indeed, Lemma \ref{synto} gives a long exact sequence of sheaves on the Fargues-Fontaine curve :
\[ \dots \to  i_{\infty,*} H^{i-1}(Z_C, L\eta_t R\nu'_* \mathbb{B}_{\rm dR}^+/ \mathrm{Fil}^r) \to \mathcal{E}^i(Z,r) \to H_{\mathcal{FF}}^i(Z)  \otimes \O_{X_{C^{\flat}}}(r) \]
\[ \to  i_{\infty,*} H^i(Z_C, L\eta_t R\nu'_* \mathbb{B}_{\rm dR}^+/ \mathrm{Fil}^r) \to \dots \]     
By Proposition \ref{criscomp}, we know that 
\[ H_{\mathcal{FF}}^i(Z) \otimes_{\O_{X_{C^{\flat}}}} \O_{X_{C^{\flat}}}(r) = \mathcal{E}(H_{\rm cris}^i(\mathfrak{Z}_s)) \otimes_{\O_{X_{C^{\flat}}}} \O_{X_{C^{\flat}}}(r). \]
This vector bundle has positive slopes, since the isocrystal\footnote{Recall that the slopes of the vector bundle are the opposites of the slopes of the isocrystal.} $H_{\rm cris}^i(\mathfrak{Z}_s)$ has slopes between $0$ and $i$ (\cite[Th. 3.1.2]{CLS}) and $i \leq r$. Hence $H_{\mathcal{FF}}^i(Z) \otimes_{\O_{X_{C^{\flat}}}} \O_{X_{C^{\flat}}}(r)$ and $H_{\mathcal{FF}}^i(Z) \otimes_{\O_{X_{C^{\flat}}}} \O_{X_{C^{\flat}}}(r)$ have no $H^1$ on the curve and the same is obviously true for the skyscraper sheaves showing up in the exact sequence. From this we deduce an exact sequence :
\[  (H_{\rm cris}^{i-1}(\mathfrak{Z}_s) \otimes B)^{\varphi=p^r} \to H^{i-1}(Z_C, L\eta_t R\nu'_* \mathbb{B}_{\rm dR}^+/ \mathrm{Fil}^r) \to \mathcal{E}^i(Z,r)  \]
\[     \to (H_{\rm cris}^{i}(\mathfrak{Z}_s) \otimes B)^{\varphi=p^r} \to H^{i}(Z_C, L\eta_t R\nu'_* \mathbb{B}_{\rm dR}^+/ \mathrm{Fil}^r).  \] 
It remains to compute the cohomology groups of $L\eta_t R\nu'_* \mathbb{B}_{\rm dR}^+/ \mathrm{Fil}^r$. We deduce from Poincar\'e's lemma (see the proof of Lemma \ref{chap13}) that :
\[ L\eta_t R\nu'_* \mathbb{B}_{\rm dR}^+ \simeq \Omega_Z^{\bullet} \hat{\otimes}_K B_{\rm dR}^+. \] 
As $Z$ is proper, this gives, for all $j\geq 0$ :
\[ H^{j}(Z_C,L\eta_t R\nu'_*  \mathbb{B}_{\rm dR}^+/ \mathrm{Fil}^r) = (H_{\rm dR}^{j}(Z) \otimes_K B_{\rm dR}^+)/\mathrm{Fil}^r. \]
This proves the claim, because of the description of geometric syntomic cohomology in terms of crystalline and de Rham cohomologies, \cite[3.2]{NN}.

But now we can reprove the \'etale-syntomic comparison theorem for free : indeed, one has a map
\[ L\eta_{\epsilon_r,t} R\nu'_* \mathbb{B} \to R\nu'_* t^r \mathbb{B} \]
which obviously induces an isomorphism between the truncations $\tau^{\leq r} L\eta_{\epsilon_r,t} R\nu'_* t^r \mathbb{B}$ and $\tau^{\leq r} R\nu'_* t^r \mathbb{B}$, by choice of the function $\epsilon_r$. But we know (Proposition \ref{proprepro\'et}) that for all $j$, 
\[ H_{\mathrm{pro\acute{e}t}}^j(Z_C,t^r \mathbb{B}) = H_{\mathrm{\acute{e}t}}^j(Z_C,\qp(r)) \otimes_{\qp} t^r B, \]
so for $i \leq r$, the sheaf $\mathcal{E}^i(Z,r)$ will just be the sheaf $H_{\mathrm{\acute{e}t}}^i(Z_C,\qp(r)) \otimes_{\qp} \O_{X_{C^{\flat}}}$. Taking global sections, we get that the above map induces an isomorphism 
\[  H_{\rm syn}^i(\mathfrak{Z}_{\O_C},r)[1/p] \simeq H_{\mathrm{\acute{e}t}}^i(Z_C,\qp(r)), \]
for every $i \leq r$. That this isomorphism is the same as the one induced by the Fontaine-Messing period map can probably be done as in \cite[\S 4.7]{CN} (the author did not check this). 
\\

\textit{Relation with the pro-\'etale cohomology of $\qp$.} Assume that $C$ is the completion of an algebraic closure of $W(k)[1/p]$. Let $Z$ be a Stein rigid space over $K$. Choose a Stein covering $(Z_k)_k$ of $Z$ (see Remark \ref{derhamlisse}).

One has an exact sequence\footnote{That a variant of such an exact sequence may be relevant to the study of the pro-\'etale cohomology of $\qp$ was suggested to the author by Gabriel Dospinescu and was used in \cite[\S 2.3]{theseaclb} to compute the pro-\'etale cohomology of $\qp$ on the affine space and the open disk.} of pro-\'etale sheaves on $Z_{C}$ :
\[ 0 \to \qp \to \mathbb{B}[1/t]^{\varphi=1} \to \mathbb{B}_{\rm dR}/\mathbb{B}_{\rm dR}^+ \to 0. \]
Assume (for simplicity) that $\dim Z=1$. The cohomology of the pro-\'etale sheaf $\mathbb{B}_{\rm dR}/\mathbb{B}_{\rm dR}^+$ can be analyzed via the Poincar\'e lemma : this was done in \cite[Rem. 2.3.21]{theseaclb}\footnote{One needs to write the Tate twists, which was not done in \textit{loc. cit.}, but is easy to do.}. One finds that $H_{\mathrm{pro\acute{et}}}^0(Z_{C}, \mathbb{B}_{\rm dR}/\mathbb{B}_{\rm dR}^+)$ is an extension of $(\O(Z_{C})/C)(-1)$ by $B_{\rm dR}/B_{\rm dR}^+$, that $H_{\mathrm{pro\acute{et}}}^1(Z_{C},\mathbb{B}_{\rm dR}/\mathbb{B}_{\rm dR}^+)= H_{\rm dR}^1(Z) \hat{\otimes}_K B_{\rm dR}/t^{-1}B_{\rm dR}^+$ and that $H_{\mathrm{pro\acute{et}}}^i(Z_{C},\mathbb{B}_{\rm dR}/\mathbb{B}_{\rm dR}^+)=0$ for $i>1$. 

One gets a short exact sequence :
\[ 0 \to (\O(Z_{C})/C)(-1) \to H_{\mathrm{pro\acute{et}}}^1(Z_{C}, \qp) \to \]
\[ \mathrm{Ker}\left(H_{\mathrm{pro\acute{et}}}^1(Z_{C}, \mathbb{B}[1/t]^{\varphi=1}) \to H_{\rm dR}^1(Z) \hat{\otimes}_K B_{\rm dR}/t^{-1}B_{\rm dR}^+ \right) \to 0. \]
Let $k\geq 0$ and $\mathcal{E} = H_{\mathcal{FF}}^1(Z_k) \otimes \O(1)$. The exact sequence, for any $k$,
\[ 0 \to H^0(X_{C^{\flat}},\mathcal{E}) \to H^0(X_{C^{\flat}} \backslash \{ \infty \}, \mathcal{E}) \to \widehat{\mathcal{E}}_{\infty}[1/t] / \widehat{\mathcal{E}}_{\infty} \to 0 \]  
tells us that 
\[ \mathrm{Ker}\left(H_{\mathrm{pro\acute{et}}}^1(Z_{C}, \mathbb{B}[1/t]^{\varphi=1}) \to H_{\rm dR}^1(Z_k)^{\dagger} \hat{\otimes}_K B_{\rm dR}/t^{-1}B_{\rm dR}^+ \right) \]
is isomorphic to
\[ H^0(X_{C^{\flat}},H_{\mathcal{FF}}^1(Z_k) \otimes \O(1))= (H^1(Z_{k,C}, \mathbb{B})^{\dagger})^{\varphi=p}. \]
Taking the inverse limit over $k$\footnote{And using that the $R^1\varprojlim$ term vanishes, as it only involves $H^0$ which can be computed.}, one gets an exact sequence :
\begin{eqnarray}
 0 \to (\O(Z_{C})/C)(-1) \to H_{\mathrm{pro\acute{et}}}^1(Z_{C}, \qp) \to H_{\mathrm{\acute{e}t}}^1(Z_{C}, L\eta_t R\nu'_* \mathbb{B})^{\varphi=p} \to 0. \label{suiteexacte} \end{eqnarray}

We now turn to the relation with the results of \cite{CDN}. Assume that $K$ is a discretely valued extension of $\qp$ with perfect residue field\footnote{This assumption is probably unnecessary.} and that $Z$ has a semi-stable formal model $\mathfrak{Z}$ over $\O_K$. \textit{If the semi-stable version of Conjecture \ref{compconj} was known to be true}, this could be rewritten using log-rigid cohomology. Indeed, let $(\mathfrak{Z}_k)_k$ be an affine covering of $\mathfrak{Z}$ such that $(Z_k)_k$ is a Stein covering of $Z$. Then we would know  that for all $k$ and all $i\geq 0$, $H_{\mathcal{FF}}^i(Z_k)$ is the vector bundle attached to the $(\varphi,N)$-module $H_{\rm logrig}^i(\mathfrak{Z}_{k,s})^{\dagger}$. In particular,
\[ H^0(X_{C^{\flat}}, H_{\mathcal{FF}}^i(Z_k)(1)) = (H_{\rm logrig}^i(\mathfrak{Z}_{k,s})^{\dagger} \otimes_{K_0} B_{\rm st}^+)^{\varphi=p,N=0}. \] 
As a Stein space is partially proper, one deduces easily that
\[ H_{\mathrm{\acute{e}t}}^1(Z_{C}, L\eta_t R\nu'_* \mathbb{B})^{\varphi=p} = (H_{\rm logrig}^1(\mathfrak{Z}_{s}) \hat{\otimes}_{K_0} B_{\rm st}^+)^{\varphi=p,N=0}. \]
Injecting this into the exact sequence \eqref{suiteexacte}, one would recover the exact sequence :
\[ 0 \to (\O(Z_{C})/C)(-1) \to H_{\mathrm{pro\acute{e}t}}^1(Z_{C}, \qp) \to (H_{\rm logrig}^1(\mathfrak{Z}_{s}) \hat{\otimes}_{K_0} B_{\rm st}^+)^{\varphi=p,N=0}  \to 0, \]
proved in \cite{CDN}.  
\\

Even without being able to express the right term of the exact sequence \eqref{suiteexacte} in terms of log-rigid cohomology, one can still say interesting things in some cases. For example, for Drinfeld's coverings of the $p$-adic upper half-plane, Cerednik-Drinfeld uniformization theorem and the Hochschild-Serre spectral sequence, combined with Remark \ref{propercase} (applied to Shimura curves), give a description of the right term of the exact sequence \eqref{suiteexacte} in terms of (classical) local Langlands and Jacquet-Langlands correspondences : this was done in \cite[\S 5]{CDN1} (their argument is formulated differently, but could be transposed to this setting).

\end{document}